\documentclass[12pt]{amsart}
\def \Z{\mathbb Z}
\def \R{\mathbb R}
\def \CC{\mathbb C}
\def \T{\mathbb T}
\def \E{\mathbb E}
\def \P{\mathbb P}
\def \N{\mathbb N}
\def \Q{\mathbb Q}
\def \A{\mathcal A}
\def \B{\mathcal B}
\def \C{\mathcal C}
\def \F{\mathcal F}

\def \I{\mathcal I}
\def \F{\mathcal F}

\newtheorem{theorem}{Theorem}[section]
\newtheorem{theorem01}{Theorem}[section]
\newtheorem*{question*}{Question}
\newtheorem*{proposition*}{Proposition}
\newtheorem*{theorem*}{Theorem}
\newtheorem{proposition}[theorem]{Proposition}
\newtheorem{corollary}[theorem]{Corollary}
\newtheorem{lemma}[theorem]{Lemma}
\theoremstyle{definition}

\newtheorem*{definitions*}{Definitions}

\newtheorem*{definition*}{Definition}
\newtheorem{definition}{Definition}
\newtheorem*{examples*}{Examples}

\newtheorem*{exmple*}{Example}
\newtheorem{exmple}[theorem]{Example} 
\newtheorem{question}{Question}

\newtheorem*{remark*}{Remark}
\newtheorem{remark}[theorem]{Remark}
\newtheorem*{remarks*}{Remarks}

\newtheorem*{note*}{Note}

\begin{document}
 
\title
{Van der Corput sets in $\Z^d$}

\author{Vitaly Bergelson and Emmanuel Lesigne}

\address[Vitaly Bergelson]{Dept. of Mathematics\\The Ohio State University\\Columbus, OH 43210, USA}
\email{vitaly@math.ohio-state.edu}

\address[Emmanuel Lesigne]{Laboratoire de Math\'ematiques et Physique Th\'eorique (UMR CNRS 6083)\\
F\'ed\'eration de Recherche Denis Poisson\\ Universit\'e
Fran\c{c}ois Rabelais\\ Parc de Grandmont, 37200 Tours, France}
\email{lesigne@univ-tours.fr}

\thanks{The first author gratefully acknowledges the support of the NSF under grant DMS-0600042}
\maketitle
\centerline{\today}
\begin{abstract}
In this partly expository paper we study van der Corput sets in $\Z^d$, with a focus on 
connections with harmonic analysis and recurrence properties of measure preserving 
dynamical systems. We prove multidimensional versions of some classical results obtained 
for $d=1$ in \cite{K-MF} and \cite{R}, establish new characterizations, introduce and 
discuss some modifications of van der Corput sets which correspond to various notions 
of recurrence, provide numerous examples and formulate some natural open questions.
\end{abstract}
\newpage
\tableofcontents
\baselineskip=17pt  

\section*{Introduction}
The main topic of our paper is the intriguing connection between 
positive-definite sequences, recurrence properties of measure preserving 
dynamical systems, and the theory of uniform distribution mod 1.

Let $(X,\A,\mu,T)$ be an invertible probability measure preserving dynamical 
system\footnote{Unless explicitly stated otherwise, we will assume in this 
paper that the measure preserving transformations we are dealing with are 
invertible and that invariant measures are normalized. We will write 
\emph{m.p.s.} for \emph{invertible probability measure preserving dynamical 
system}.}. Given a set $A\in\A$ with $\mu(A)>0$, let $R_A=\{n\in\Z,\,n\neq~0\,:\,
\mu\left(A\cap T^nA\right)>0\}$. While the classical Poincar\'e recurrence 
theorem, which states that the set $R_A$ is non-empty (and hence infinite), 
is nowadays an easy exercise, quite a few of
the more subtle properties of \emph{sets of 
returns} $R_A$ and of related sets $R_{A,\epsilon}=\{n\in\Z,\, n\neq0\,:\,
\mu\left(A\cap T^nA\right)>\epsilon\}$ are still not fully understood.

Following Furstenberg (\cite{HF}), let us call a set of integers 
$D$ a \emph{set of
recurrence}, if for any m.p.s. $(X,\A,\mu,T)$ and any $A\in\A$ with $\mu(A)>0$ one 
has $D\cap
R_A\neq\emptyset$. For example, for any $k\in\N$, the set $k\N$ is a set of recurrence 
(just consider the system $(X,\A,\mu,T^k)$) and any set of recurrence has a non-empty 
intersection with the set $k\N$ (just consider a permutation of a finite set). 
A
more general (and still rather trivial) example is provided by the set of differences 
$\{n_i-n_j\,:\,i>j\}$, where $(n_i)_{i\geq1}$ is an increasing sequence of integers. 
(To
see that this is a set of recurrence, just observe that if $\mu(A)>0$, then 
the sets $T^{n_i}A$ cannot be pairwise 
disjoint, $\mu(X)$ being finite.) The following generalization of the Poincar\'e
recurrence 
theorem obtained by Furstenberg (see \cite{HF1}, \cite{HF}) gives a much less trivial example of a set of recurrence.
\begin{theorem01}
For any polynomial $p(n)\in\Z[n]$, satisfying $p(0)=0$, for any m.p.s. $(X,\A,\mu,T)$ and for any $A\in\A$ with $\mu(A)>0$, there 
exists $n\in\N$ such that $p(n)\neq0$ and $\mu\left(A\cap T^{p(n)}A\right)>0$.
\end{theorem01}

Following Ruzsa (\cite{R}), let us call a set $D\subset\N$ \emph{intersective}, if for 
any  $S\subset\N$ of positive upper density\footnote{The subset $S$ of $\N$ has 
\emph{positive upper density} if $$\overline d(S):=\limsup_{N\to+\infty}\frac1N
\left|S\cap\{1,2,\ldots,N\}\right|>0.$$} there exist $x,y\in S$ such 
that $x-y\in D$. It is not hard to show that a set $D$ is intersective if and only 
if it is a set of \emph{combinatorial recurrence}, that is, such that 
for any $S\subset\N$ with $\overline d(S)>0$, there exists $n\in D$ such that 
$\overline d\left(S\cap(S-n)\right)>0$. This hints that the notions ``set of recurrence''
and ``intersective set'' are related and, indeed, it turns out 
that these notions coincide. (The fact that intersectivity implies measure-theoretic recurrence has been remarked by several authors, see for example \cite{Anne} 
and \cite{bergelson}. The fact that measure-theoretic recurrence implies combinatorial recurrence is a consequence of \emph{Furstenberg's correspondence principle}, see for example \cite{VB3}.)

Thus, for example, Theorem 0.1 implies S\`ark\"ozy's theorem (\cite{sarko}), 
which states that for any polynomial $p(n)\in\Z[n]$ satisfying $p(0)=0$ and any set 
$S\subset\N$ with $\overline d(S)>0$ there exist $x,y\in S$ and $n\in\N$ 
such that $x-y=p(n)$.
We remark that it was shown in \cite{K-MF} that a necessary and sufficient condition 
for a polynomial $p(n)\in\Z[n]$ to satisfy the 
Furstenberg-S\`ark\"ozy theorem is that for any positive integer $k$ there exists an 
integer $n$ such that $p(n)$ is divisible by $k$. Actually, Kamae and Mend\`es France 
in \cite{K-MF} showed that many sets of recurrence, including the mentioned above sets 
have a stronger property which they called the \emph{van der Corput property}. 

\begin{definition*}
A set $D$ of positive integers is a van der Corput set (or vdC set) if it has the
following 
property : given a real sequence $(x_n)_{n\in\N}$, if all the sequences $(x_{n+d}-
x_n)_{n\in\N},\,d\in D,$ are uniformly distributed mod 1, then the sequence 
$(x_n)_{n\in\N}$ is itself uniformly distributed mod 1. \end{definition*}

This concept and terminology\footnote{Ruzsa uses the name \emph{correlative set} 
instead of van der Corput set.} come from the \emph{van der Corput inequality}, which 
is presented at the beginning of the next section, and which motivates the following
\emph{van der Corput trick}: if for a given real sequence $(x_n)_{n\in\N}$ and any 
$h\in\N$ the sequence $(x_{n+h}-x_n)_{n\in\N}$ is uniformly distributed mod 1, then 
the sequence $(x_n)_{n\in\N}$ is uniformly distributed mod 1. Van der Corput's 
inequality and its application to uniform distribution appeared for the first time 
in \cite{vdc}, under the name \emph{Dritte Haupteigenschaft} (third principal property).

\medbreak
Kamae and Mend\`es France showed in \cite{K-MF} that every vdC set is a set of recurrence. The other implication is false : Bourgain has constructed in \cite{JB} an example of a set of recurrence which is not a vdC set.

The notions introduced above are connected via the notion of positive-definiteness. 
Indeed, it is easy to check that the sequence $\left(\mu(A\cap T^nA)\right)$  is positive-definite\footnote{This fact was first noticed and utilized by Khintchine in \cite{Kh}.}, which establishes the connection between sets of recurrence and properties of positive-definite sequences. As for the vdC property, let us first note that in light of Weyl's criterion (see \cite{KN}), the sequence $(x_{n+d}-x_n)_{n\in\N}$ is uniformly distributed mod 1 if and only if, for any $k\in\Z,\,k\neq0,$ one has
\begin{equation}\label{Wc}
\lim_{N\to+\infty}\frac1N\sum_{n=1}^N e^{2\pi ik(x_{n+d}-x_n)}=
\lim_{N\to+\infty}\frac1N\sum_{n=1}^N e^{2\pi ikx_{n+d}}\overline{e^{2\pi ikx_n}}=0.
\end{equation}
Now, given a bounded sequence $\alpha:\N\rightarrow\CC$, it is not hard to see that 
for some increasing sequence of integers $(N_j)_{j\in\N}$ the limit
\begin{equation}\label{correl}
\lim_{j\to+\infty}\frac1{N_{j}}\sum_{n=1}^{N_{j}} 
\alpha(n+d)\overline{\alpha(n)}
=\gamma(d)
\end{equation}
exists for all $d\in\Z$ and that, moreover, the sequence $\gamma$ is positive-definite (see \cite{bertrandias}). Juxtaposing (\ref{Wc}) and (\ref{correl}) we see that the vdC property is also connected to the properties of positive-definite sequences.

By the Bochner-Herglotz theorem (see for example \cite{rudin}, Subsection 1.4.3), any 
positive-definite sequence $\varphi$ is given by the 
Fourier coefficients of a positive measure $\nu_\varphi$ on the circle :
$$
\varphi(n)=\int_{\T} e^{2\pi inx}\, \text{d}\nu_\varphi(x)\;,
$$
and the properties of this measure play a crucial role in verifying that certain sets are vdC and in establishing the connections between (various versions of) vdC sets and sets of recurrence (see in particular Section \ref{vdc-recur} below ).

The following fact is also useful for a better understanding of the link between vdC sets and
sets of recurrence. Let $D\subset\Z$. We prove (see Corollary \ref{cor.gen.vdC.ineq.hilb})
that $D$ is a vdC set if and only if the following is true : given a bounded sequence
of complex numbers
$(u_{n})_{n\in\N}$, if for all $d\in D$, the sequence $(u_{n+d}\overline{u_{n}})$ converges
to zero in the Ces\`aro sense, then the sequence $(u_{n})$ also converges to zero in the Ces\`aro
sense. We also prove (see Theorem~\ref{vdc01}) that $D$ is a set of recurrence if
and only if the analagous property holds with ``$(u_{n})$ is a bounded sequence of complex
numbers'' replaced by  ``$(u_{n})$ is a bounded sequence of positive real
numbers''.\\

Driven by the desire to obtain new applications to combinatorics and to better understand 
the recurrence properties of measure-preserving $\Z^d$-actions, we focus in this 
paper on $\Z^d$ versions of vdC sets. As we will see, many known properties extend from 
$\Z$  to $\Z^d$ with relative ease. Still, some properties turn out to be more recalcitrant 
and their extensions to $\Z^d$ demand more work.

The definition of vdC set in $\Z^d$ is given in Subsection \ref{vdc-sec}.
Here are some examples of facts/theorems which will be obtained in subsequent sections.\begin{itemize}
\item 
The class of vdC sets has the Ramsey property. Namely, If $D$ is a vdC set in $\Z^{d}$ and if 
$D=D_{1}\cup D_{2}$ then at least one
        of the $D_{i}$ is a vdC set. 
\item
Let $p_1,p_2,\ldots,p_d$ be a finite family of polynomials with integer coefficients, to which we associate the subset $S=\{(p_1(n),p_2(n),.\,.\,.,p_d(n)):n\in\N\}$ of $\Z^d$. The following properties are equivalent :
\\
- The set $S$ is a set of recurrence for $\Z^d$-actions\footnote{A subset $S$ of
$\Z^d$ is called a set of recurrence for $\Z^{d}$-actions if, given any 
measure preserving $\Z^d$-action $\left(T_{n}\right)_{n\in\Z^d}$ on a probability 
space $(X,\A,\mu)$ and any $A\in\A$ with 
$\mu(A)>0$, there
exists $n\in S,\,n\neq0$ such that 
$\mu\left(A\cap T_{n}A\right)>0$.}.
\\
- The set $S$ is a vdC set in $\Z^d$.
\\
- The set $S$ is a set of multiple recurrence for 
$\Z$-actions\footnote{A subset $S$ of
$\Z^d$ is called a set of multiple recurrence for $\Z$-actions if, given any m.p.s. $(X,\A,\mu,T)$ and any $A\in\A$ with 
$\mu(A)>0$, there
exists $(n_1,n_2,\ldots,n_d)\in S\setminus\{(0,0,\ldots,0)\}$ such that 
$\mu\left(A\cap T^{n_1}A\cap T^{n_2}A\cap\ldots\cap T^{n_d}A\right)>0$.}.
\\
- For any $q\in\N$, there exists $n\in\N$ such that $p_1(n), p_2(n),\ldots, p_d(n)$ are all divisible by $q$.\\
Moreover these equivalent properties are also necessary and sufficient for the set $S$ to be an
\emph{enhanced vdC set} (see Definition \ref{enhanced-def} in 
Subsection \ref{enhanced-def-sec}) and a set of \emph{strong recurrence} (see Definition
\ref{rec-def} in Subsection
\ref{def-rec-sec}).
\item
Let $P$ be the set of prime numbers. For any finite family $f_{1}, f_{2}, \ldots f_{d}$ 
of polynomials with integer coefficients and with zero constant terms the set $\{f_{1}(p-1),f_{2}(p-1),\ldots,f_{d}(p-1)\,:
\,p\in P\}$ is a vdC set in $\Z^d$. (It can also be proved that it is an enhanced 
vdC set; see below.)
\item

The Cartesian product of two vdC sets is a vdC set in the corresponding product of parameters space.
\item
A subset $D$ of $\Z$ is a vdC set if and only if
        any positive measure $\sigma$ on the torus $\T$ such
        that
$
\sum_{d\in D} \left|\widehat\sigma(d)\right|<+\infty
$  is continuous.
\item
We establish a \emph{generalized van der Corput inequality} for multiparameter sequences in a Hilbert space (Proposition \ref{gen.vdC.ineq.hilb}).
\end{itemize}

In order to make the paper more readable we will restrict discussion 
mainly to dimension $d=2$. The reader should have no problem verifying that our proofs 
work for general $d\in\N$.

In Section \ref{enhancedvdcset}, we introduce the notion of ``enhanced vdC set''. 
We show that the enhanced vdC property
is equivalent to the FC$^+$ property (which appears in \cite{K-MF}, with a reference to
Y. Katznelson). Moreover, the enhanced vdC property is related to the notion of strong recurrence in the same way as vdC sets are related to sets of recurrence. In Subsection \ref{enhancedquestions} we collect some natural open questions.

In Section \ref{vdc-recur} we discuss links between recurrence and vdC properties. We also 
introduce and discuss the notions of \emph{density vdC set} and \emph{nice vdC set}.

In Section \ref{dist-notion} we briefly discuss some modifications of the notion of vdC set which are connected to various notions of uniform distribution.
\medbreak
It is worth mentioning that in practically 
every paper in the area of Ergodic Ramsey Theory, some version of the van der Corput trick 
for sequences in Hilbert spaces is used. See for example 
\cite{FKO}, \cite{BL1}, \cite{BLM}, \cite{vitaly-randall}, \cite{flw} dealing with 
multiple recurrence, and \cite{pet}, \cite{BL2}, \cite{host-kra}, \cite{tamar} and
\cite{sacha-pol} dealing with mean convergence of multiple ergodic averages. The van der 
Corput trick is also useful in establishing results pertaining to pointwise convergence : see for example \cite{lesigne} and \cite{nikos}.  

\medbreak

The influence on our work of the above-mentioned paper of Kamae and Mend\`es France, and 
of the fundamental ideas developed by Ruzsa in \cite{R}, cannot be exaggerated. We are 
especially grateful to Randall McCutcheon for numerous useful suggestions, and would like
to thank Inger H\aa land-Knutson, Anthony Quas and M\'at\'e Wierdl for pertinent
communications.

\medbreak
Throughout the paper, we will use the classical notation $e(t):=e^{2\pi it}$ 
for $t\in\R$ or $t\in\T=\R/\Z$.

\section{Van der Corput sets in $\Z^{d}$}

In this section we develop a theory of van der Corput sets in the
multidimensional lattice $\Z^{d}$, which is parallel to the known theory in
$\Z$ (see \cite{K-MF}, \cite{R}, \cite{Mont}). As we have already mentioned in the 
introduction, we limit our presentation to the
case $d=2$. Definitions, results and arguments in this section follow the one
dimensional case, except at one point : in order to obtain a
\emph{generalized van der Corput inequality}, Ruzsa uses in \cite{R} a theorem of
Fejer stating that any positive trigonometric polynomial in one
variable is the square modulus of another trigonometric polynomial ;
this fact is no longer true for trigonometric polynomial of several
variables, hence we are forced to use a different 
argument to derive the generalized van
der Corput inequality in the multidimensional case (cf. Subsection
\ref{inequality}).

\subsection{Van der Corput's inequality and van der Corput's principle}
\subsubsection{Van der Corput's inequality in $\Z^{2}$}{\ }\\
For $a,b,c,d\in\Z$, we write $(a,b)\leq(c,d)$ if $a\leq c$ and $b\leq d$.
(Similarly for $<$, $\geq$ and $>$.) We write $0$ for $(0,0)\in\Z^2$.

\begin{theorem}\label{premier}
Let $N=(N_{1},N_{2})\in\N^{2}$, and $(u_{n})_{0< n\leq N}$ be a finite
family of complex numbers indexed by
$\left([1,N_{1}]\times[1,N_{2}]\right)\cap\Z^{2}$.

For $ h=(h_{1},h_{2})\in \Z^{2}$, we define
$$
\gamma(N,h):=\sum_{\begin{subarray}{1}0< n\leq N\\0< n+h\leq N\end{subarray}}
u_{n+h}\cdot\overline{u_{n}}\;.
$$

For any $H=(H_{1},H_{2})\in\N^{2}$, we have
$$
\Big|\sum_{0< n\leq N}u_{n}\Big|^{2}\leq
\frac{(N_{1}+H_{1})(N_{2}+H_{2})}{H_{1}^{2}H_{2}^{2}}
\sum_{-H<h<H}(H_{1}-|h_{1}|)(H_{2}-|h_{2}|)\gamma(N,h)\;.
$$
\end{theorem}
The preceding inequality is usually used in the following form
\begin{equation}\label{vdci}
\Big|\sum_{0<n\leq N}u_{n}\Big|^{2}\leq
\frac{(N_{1}+H_{1})(N_{2}+H_{2})}{H_{1}H_{2}}
\sum_{-H<h<H}\left|\gamma(N,h)\right|\;.
\end{equation}
(The proof of Theorem \ref{premier} is an elementary application of Cauchy's inequality. It is a particular case of the calculations presented in Subsection \ref{abstract}.)
\subsubsection{Van der Corput's principle in $\Z^{2}$}{\ }
Let $(u_{n})_{n\in\N^{2}}$ be a family of complex numbers. Starting from inequality (\ref{vdci}), dividing by $(N_1N_2)^2$, then letting $N_1$ and $N_2$ go to infinity, we obtain that, for any $H\in\N^2$,
$$
\limsup_{N_{1},N_{2}\to+\infty}\Big|\frac1{N_1N_2}\sum_{0<n\leq N}u_{n}\Big|^{2}\leq\frac1{H_{1}H_{2}}\sum_{-H<h<H}
\left(\limsup_{N_{1},N_{2}\to+\infty}
\frac1{N_{1}N_{2}}|\gamma(N,h)|\right)\;.
$$
As a direct consequence we obtain the following proposition.
\begin{proposition}\label{vdc-p}
If $(u_{n})_{n\in\N^{2}}$ is a family of complex numbers such that
$$
\inf_{H>0}\frac1{H_{1}H_{2}}\sum_{-H<h<H}
\left(\limsup_{N_{1},N_{2}\to+\infty}
\frac1{N_{1}N_{2}}|\gamma(N,h)|\right)=0
$$
then
$$
\lim_{N_{1},N_{2}\to+\infty}\frac1{N_{1}N_{2}}\sum_{0<n\leq N}u_{n}=0\;.
$$
\end{proposition}
We use the following notion of uniform distribution for a family
indexed by $\N^{2}$.
\begin{definition} A family $(x_{n})_{n\in\N^{2}}$ of real numbers
        is uniformly
distributed mod 1 if for any continuous function $f$ on $\R$,
invariant under translations by elements of $\Z$, we have
\begin{equation}\label{moyennes}
\lim_{N_{1},N_{2}\to+\infty}\frac1{N_{1}N_{2}}\sum_{0<n\leq N}
f(x_{n}) =\int_{[0,1]}f(t)\ \mbox{d}t\;.
\end{equation}
\end{definition}
Other useful notions of uniform distribution can be introduced : for example, one can 
replace in (\ref{moyennes}) the averages $\left(\frac1{N_{1}N_{2}}\sum_{0<n\leq N}
\ldots\right)_{N_1,N_2\to+\infty}$ by \\$\left(\frac1{(N_{1}-M_1)(N_{2}-M_2)}
\sum_{M\leq n<N}\ldots\right)_{N_1-M_1,N_2-M_2\to+\infty}$;
this leads to the notion of \emph{well distributed sequences}. Or, one can consider 
averages defined by a given F\o lner sequence. We postpone remarks on these variations 
to Section~\ref{dist-notion}.

Note that since property (\ref{moyennes}) has an asymptotic nature, it makes sense
even if the entries in the sequence $(x_n)$ are defined only for
indices $n = (n_1,n_2)$ for $n_1, n_2$ large enough. We tacitly
utilize this observation in the formulation of Corollary \ref{standard vdC}
below and throughout the paper.

Let us recall the classical Weyl's criterion for uniform distribution (see \cite{weyl}, \cite{KN}). A family $(x_{n})_{n\in\N^{2}}$ of real numbers
        is u.d. mod $1$ if and only if,
        for any $k\in\Z\setminus\{0\}$,
        $$
        \lim_{N_{1},N_{2}\to+\infty}\frac1{N_{1}N_{2}}\sum_{0<n\leq N}
e(kx_{n}) =0\;.
$$

As in dimension 1, van der Corput's principle in $\Z^d$ has a useful corollary pertaining to uniform distribution.
\begin{corollary}\label{standard vdC}
Let $(x_{n})_{n\in\N^{2}}$ be a family of real numbers. If for any
$h\in\Z^{2}\setminus\{0\}$ the family $(x_{n+h}-x_{n})_{n\in\N^{2}}$
is u.d.~mod $1$, then the family $(x_{n})_{n\in\N^{2}}$
is u.d.~mod $1$.
\end{corollary}
When we apply Proposition \ref{vdc-p} in order to prove 
Corollary \ref{standard vdC}, we see that it is sufficient to let only one of 
$H_{1},H_{2}$ go to
infinity. The following definition will allow us to give a more
general version of this corollary.

Let $D$ be a subset of $\Z^{2}$. We define
$$
\delta(D):=\sup_{H_{1},H_{2}\geq0}
\frac1{(2H_{1}+1)(2H_{2}+1)}\mbox{ card}
\left(D\cap[-H_{1},H_{1}]\times[-H_{2},H_{2}]\right)\;.
$$
(Note that $\delta(D)$ is \emph{not} the ordinary notion of density, which corresponds to\\ $\limsup_{\min\{H_{1},H_{2}\}\to+\infty}$.)
\begin{corollary}\label{firstexample}
Let $(x_{n})_{n\in\N^{2}}$ be a family of real numbers, and
$D\subset\Z^{2}\setminus\{0\}$. If $\delta(D)=1$ and if, for any $d\in D$, the
family $(x_{n+d}-x_{n})$ is u.d. mod 1, then
the family $(x_{n})$ is u.d. mod 1.
\end{corollary}

\begin{proof}
There exists a sequence $\left(H^{(k)}\right)$ (with $H^{(k)}
:=(H_{1}^ {(k)},H_{2}^{(k)})$) in $\left(\N\cup\{0\}\right)^{2}$
such that 
$$
\lim_{k\to+\infty}\frac1{(2H_{1}^{(k)}+1)(2H_{2}^{(k)}+1)}\mbox{ card}
\left(D\cap[-H_{1}^{(k)},H_{1}^{(k)}]\times[-H_{2}^{(k)},H_{2}^{(k)}]\right)=1\;.
$$
Let $(u_{n})_{n\in\N^2}$ be a family of complex numbers of modulus $1$ such that,
for any $d\in D$,
$$
\lim_{N_{1},N_{2}\to+\infty}\frac1{N_{1}N_{2}}\sum_{0< n\leq N}
u_{n+d}\overline{u_{n}}=0\;.
$$
For any $d\in D$, we have
$$\lim_{N_{1},N_{2}\to+\infty}\frac1{N_{1}N_{2}}\gamma(N,d)=0\;.$$
  We deduce from van der Corput's inequality that
$$
\Big|\frac1{N_{1}N_{2}}\sum_{0< n\leq N}u_{n}\Big|^{2}\leq
\frac{(N_{1}+H_{1}+1)(N_{2}+H_{2}+1)}{N_{1}N_{2}(H_{1}+1)(H_{2}+1)}
\sum_{-H\leq d\leq H}\frac1{N_{1}N_{2}}|\gamma(N,d)|\;.
$$
Using the fact that $|\gamma(N,d)|\leq N_{1}N_{2}$, we obtain
\begin{multline*}
\limsup_{N_{1},N_{2}\to+\infty}
\Big|\frac1{N_{1}N_{2}}\sum_{0<n\leq N}u_{n}\Big|^{2}\\
\leq
\frac1{(H_{1}+1)(H_{2}+1)}
\mbox{card}\left(D^c\cap[-H_{1},H_{1}]\times[-H_{2},H_{2}]\right)\;.
\end{multline*}
The right hand side of the last inequality goes to zero along the sequence $\left(H^{(k)}\right)$. This argument can be applied to $u_n=e(kx_n)$ (no matter how $x_n$ is defined for $n\in\Z^2\setminus\N^2$) for any choice of $k\in\Z,\,k\neq0$.
Thus, the result follows from Weyl's criterion.
\end{proof}

\begin{exmple*} If, for any positive integer $j$, the family
$(x_{n+(j,0)}-x_{n})$ is u.d.~mod 1, then the family $(x_{n})$ is
u.d.~mod $1$. 
\end{exmple*}
\begin{exmple*} 
The first application of van der Corput's inequality was to Weyl's equidistribution theorem for polynomial sequences (\cite{weyl}, \cite{vdc}). The two-parameter version of this theorem says the following : if $P\in\R[X,Y]$ is a real polynomial in two variables and if at least one coefficient of a non constant monomial in $P$ is irrational, then the family $\left(P(n_1,n_2)\right)_{(n_1,n_2)\in\N^2}$ is uniformly distributed mod 1. (This result has a straightforward generalization to polynomials in more than two variables.) This multiparameter equidistribution theorem is a direct consequence of either Corollary \ref{standard vdC}, or Corollary \ref{firstexample} applied to sets $D={0}\times\N$ and $D=\N\times{0}$.
\end{exmple*}

\subsubsection{An abstract version of van der Corput's principle}{\ }
\label{abstract}
\begin{proposition}
Let $(G,\cdot)$ be a group, and $E$, $D$ two finite subsets of $G$. Let $u$
be a complex-valued function defined on $E$. We have
\begin{equation}\label{vdca}
\left|\sum_{n\in E}u(n)\right|^{2}\leq \frac{|E\cdot D^{-1}|}{|D|}\sum_{d\in
D\cdot D^{-1}}\
\left|\sum_{\begin{subarray}{1}n\in E\\n\in E\cdot d^{-1}\end{subarray}}
u(n\cdot d)\overline{u(n)}\right|\;.
\end{equation}
\end{proposition}

\begin{proof}
Define $u(n)$ to be zero if $n\notin E$. We have
$$
\left|\sum_{n\in E}u(n)\right|^{2}=\left|\frac1{|D|}\sum_{d\in D}\
\sum_{n\in E\cdot d^{-1}}u(n\cdot d)\right|^{2}=\left|\frac1{|D|}
\sum_{n\in E\cdot D^{-1}}\ \sum_{d\in D}u(n\cdot d)\right|^{2}\;.
$$
Using Cauchy's inequality, we obtain
$$
\left|\sum_{n\in E}u(n)\right|^{2}\leq\frac{|E\cdot D^{-1}|}{|D|^{2}}
\sum_{n\in G}\left|\sum_{d\in D}u(n\cdot d)\right|^{2}\;,
$$
and this last expression is equal to
\begin{multline*}
\frac{|E\cdot D^{-1}|}{|D|^{2}}\sum_{d,d'\in D}\ \sum_{n\in G}
u(n\cdot d)\overline{u(n\cdot d')}\\=
\frac{|E\cdot D^{-1}|}{|D|^{2}}\sum_{d'\in D}\ \sum_{d\in D\cdot {d'}^{-1}}\
\sum_{n\in G}u(n\cdot d)\overline{u(n)}\\
\leq\frac{|E\cdot D^{-1}|}{|D|^{2}}\sum_{d'\in D}\ \sum_{d\in D\cdot
D^{-1}}\
\left|\sum_{n\in G}u(n\cdot d)\overline{u(n)}\right|\;.
\end{multline*}

\end{proof}
Note that inequality (\ref{vdca}) contains inequality (\ref{vdci}) as a special case corresponding to \\$G=\Z^2$, $E=\left([1,N_1]\times[1,N_2]\right)\cap\Z^2$ and $D=\left([1,H_1]\times[1,H_2]\right)\cap\Z^2$.

\begin{remark}
The vdC inequality that has been stated above for a family of complex numbers can be
extended verbatim to any family of vectors in a linear complex space equipped with a
scalar product. This fact is very useful in many applications to mean convergence theorems
or recurrence theorems in Ergodic Theory (see for example Lemma A6 and the references in \cite{vitaly-randall}).
\end{remark}
\subsection{Van der Corput sets}\label{vdc-sec}
\subsubsection{Definition}\label{definition}
\begin{definition}\label{vdc-def}
A subset $D$ of $\Z^{2}\setminus\{0\}$ is a {\em van der Corput set}
(vdC-set)
if for any family $(u_{n})_{n\in\Z^2}$ of complex numbers of modulus
1 such that
$$
\forall d\in D,\quad\lim_{N_{1},N_{2}\to+\infty}
\frac1{N_{1}N_{2}}\sum_{0\leq n<(N_{1},N_{2})}u_{n+d}\overline{u_{n}}=0
$$
we have
\begin{equation}\label{Ces_dim2}
\lim_{N_{1},N_{2}\to+\infty} \frac1{N_{1}N_{2}}\sum_{0\leq
n<(N_{1},N_{2})}u_{n}=0\;.
\end{equation}
\end{definition}
Equivalently, $D$ is a vdC-set if any family 
$(x_{n})_{n\in\N^{2}}$ of real numbers having the property that for all $d\in D$ the family
$(x_{n+d}-x_{n})_{n\in\N^2}$ is u.d. mod 1, is itself u.d. mod 1.

(A natural Ces\`aro summation method is also given by ``bilateral averages''. One 
obtains an equivalent definition of vdC set if we replace in Definition \ref{vdc-def} sums $\sum_{0\leq
n<(N_{1},N_{2})}$ by sums $\sum_{(-N_1,-N_2)<
n<(N_{1},N_{2})}$. See Section \ref{dist-notion}.)

\begin{exmple}If $\delta(D)=1$, the set $D$ is a vdC-set (see Corollary
\ref{firstexample}).
\end{exmple}
Note that various modifications of the notion of uniform distribution (for example, considering other types of averages) lead, generally speaking, to different notions of vdC set. See Section \ref{dist-notion} for some remarks and open questions.

\subsubsection{Spectral characterization}\label{spectral-char1}
If $\sigma$ is a finite measure on the 2-torus $\T^2$, we define its Fourier transform $\widehat\sigma$ by $\widehat\sigma(n)=\int_{\T^2}e(n_1x_1+n_2x_2)\,\text{d}\sigma(x_1,x_2)$, for any $n=(n_1,n_2)\in\Z^2$.
\begin{theorem}\label{spectralchar}
Let $D\subset \Z^{2}\setminus\{0\}$. The following statements are equivalent
\begin{enumerate}
        \item[(S1)]
        $D$ is a van der Corput set.
        \item[(S2)]
        If $\sigma$ is a positive measure on the 2-torus $\T^{2}$ such
        that, for all $d\in D$, $\widehat\sigma(d)=0$, then
        $\sigma\left(\{(0,0)\}\right)=0$.
        \item[(S3)]
        If $\sigma$ is a positive measure on the 2-torus $\T^{2}$ such
        that, for all $d\in D$, $\widehat\sigma(d)=0$, then $\sigma$ is
        continuous.
        \end{enumerate}
\end{theorem}

(Note that we prove in the sequel (Subsection \ref{new-spectral}) that (S1), (S2) and (S3) are equivalent to the following property : 
any positive measure $\sigma$ on the 2-torus $\T^2$ such
        that
$
\sum_{d\in D} \left|\widehat\sigma(d)\right|<+\infty
$  is continuous.)

The equivalence of (S2) and (S3) is clear, since a translation of a
measure does not change the modulus of its Fourier coefficients. For one dimensional space of parameters the
implication (S2)$\Rightarrow$(S1) is proved in
\cite{K-MF} and the implication  (S1)$\Rightarrow$(S2) can be found in \cite{R}.

\begin{lemma}\label{class_spect}
Let $(u_{n})_{n\in\Z^{2}}$ be a bounded family of complex numbers
and $(N^{(j)})_{j\in\N}=\big((N^{(j)}_{1},N^{(j)}_{2})\big)_{j\in\N}$ be a
sequence in $\N^{2}$ such that $\min(N^{(j)}_{1},N^{(j)}_{2})\to+\infty$ when
$j\to+\infty$. If, for all $h\in\Z^{2}$,
$$
\gamma(h):=\lim_{j\to+\infty}\frac1{N^{(j)}_{1}N^{(j)}_{2}}\sum_{0\leq
n<N^{(j)}}u_{n+h}\overline{u_{n}}\qquad\mbox{exists}\;,
$$
then there exists a positive measure $\sigma$ on the 2-torus $\T^{2}$
such that, for all $h\in\Z^{2}$,
$$
\widehat\sigma(h)=\gamma(h)
$$
and this measure satisfies
$$
\limsup_{j\to+\infty}\frac1{N^{(j)}_{1}N^{(j)}_{2}}\left|\sum_{0\leq
n<N^{(j)}}u_{n}\right|\leq\sqrt{\sigma\left(\{(0,0)\}\right)}\;.
$$
\end{lemma}

\begin{proof}[Sketch of the proof of Lemma \ref{class_spect}]
We denote $x=(x_{1},x_{2})$, $n=(n_{1},n_{2})$, etc\ldots

The family $(\gamma_{h})_{h\in\Z^2}$ is positive-definite and the
Bochner-Herglotz Theorem guarantees the existence of the positive
measure $\sigma$ (see for example \cite{rudin}, Subsection 1.4.3). This measure is the weak limit of the sequence of
absolutely continuous measures $(\sigma_{N^{(j)}})$ where
$\sigma_{N}$ has density
$$
g_{N}(x):=\frac1{N_{1}N_{2}}\left|\sum_{0\leq n<N}
u_{n}e(-n_{1}x_{1}-n_{2}x_{2})\right|^{2}
$$
with respect to Lebesgue measure d$x_{1}$d$x_{2}$.

We define
$$
h_{N}(x):=\frac1{N_{1}N_{2}}\left|\sum_{0\leq n<N}
e(-n_{1}x_{1}-n_{2}x_{2})\right|^{2}\;.
$$
The sequence of measures with density $h_{N}$ converges weakly to the
Dirac delta measure at $(0,0)$, denoted by $\delta$.

We follow the method of \cite{CKMF}, in particular their Theorem 2, which utilizes the 
connection between the \emph{affinity}\footnote{Let $\mu$ and $\nu$ be two probability 
measures on $\T^2$. The affinity $\rho(\mu,\nu)$ is defined as
$$
\rho(\mu,\nu) = \int_{\T^2}\left(\frac{\text{d}\mu}{\text{d}m}\right)^{1/2}\left(\frac{\text{d}\nu}{\text{d}m}\right)^{1/2}\ \text{d}m\;,
$$
where $m$ is any measure with respect to which both $\mu$ and $\nu$ are absolutely 
continuous. Note that affinity is also called the \emph{Hellinger 
integral} by probabilists. It is proved in \cite{CKMF} that if $(\mu_n)$ and $(\nu_n)$ 
are two weakly convergent sequences of probability measures, 
then $$\limsup_{n\to+\infty}\rho(\mu_n,\nu_n)\leq\rho(\lim\mu_n,\lim\nu_n).$$
} of two probability measures and 
weak convergence. Denoting by $\rho(\mu,\nu)$ the affinity of two probability measures 
on $\T^2$, we have
$$\rho\left(g_N(x)\text{d}x, h_N(x)\text{d}x\right)=\int_{\T^{2}}\sqrt{g_{N}(x)h_{N}(x)}\
\mbox{d}x_{1}\mbox{d}x_{2}\;,
$$
$$
\rho(\sigma,\delta)=\sqrt{\sigma(\{(0,0)\})}\;,
$$
and
$$
\limsup_{j\to+\infty}\int_{\T^{2}}\sqrt{g_{N^{(j)}}(x)h_{N^{(j)}}(x)}\
\mbox{d}x_{1}\mbox{d}x_{2}\leq\sqrt{\sigma\left(\{(0,0)\}\right)}\;.
$$
The conclusion of the lemma then follows from the inequality
$$
\frac1{N^{(j)}_{1}N^{(j)}_{2}}\left|\sum_{0\leq
n<N^{(j)}}u_{n}\right|\leq\int_{\T^{2}}\sqrt{g_{N^{(j)}}(x)h_{N^{(j)}}(x)}\
\mbox{d}x_{1}\mbox{d}x_{2}\;.
$$
\end{proof}
\begin{proof} [Proof of Theorem \ref{spectralchar}]
Let us first prove that (S2)$\Rightarrow$(S1). Let $(u_{n})_{n\in\Z^{2}}$
be a bounded family of complex numbers such that, for all $d\in D$,
$$
\lim_{N_{1},N_{2}\to+\infty}
\frac1{N_{1}N_{2}}\sum_{0\leq
n<(N_{1},N_{2})}u_{n+d}\overline{u_{n}}=0\;.
$$
There exists a sequence $(N^{(j)})_{j\in\N}$ in $\N^{2}$ such that
\begin{itemize}\item
        $\min(N^{(j)}_{1},N^{(j)}_{2})\to+\infty$\;,
        \item
        $\displaystyle\lim_{j\to+\infty}\frac1{N^{(j)}_{1}N^{(j)}_{2}}\left|\sum_{0\leq
n<N^{(j)}}u_{n}\right|=
\limsup_{N_{1},N_{2}\to+\infty} \frac1{N_{1}N_{2}}\left|\sum_{0\leq
n<N}u_{n}\right|$\;,
\item
$\displaystyle\forall h\in \Z^2,\quad \gamma(h):=
\lim_{j\to+\infty}\frac1{N^{(j)}_{1}N^{(j)}_{2}}\sum_{0\leq
n<N^{(j)}}u_{n+h}\overline{u_{n}}\qquad\mbox{exists}$\;.
\end{itemize}
The map $\gamma$ is the Fourier transform of a positive measure
$\sigma$ on the $2$-torus. We have $\widehat\sigma(d)=0$ for all $d\in
D$. By condition (S2), the measure $\sigma$ has no point mass at $(0,0)$,
and, using Lemma \ref{class_spect}, we conclude that the family $(u_{n})_{n\in\Z^{2}}$ converges to zero in the sense of (\ref{Ces_dim2}). We have proved that $D$ is a vdC-set.

Following Ruzsa (\cite{R}), we will use a probabilistic argument in order to
prove that (S1)$\Rightarrow$(S2). The next two lemmas are routine
variations on the theme of the law of large numbers.

\begin{lemma}\label{lemmeprob.1}
        Let $(\theta(n))_{n\in\N^{2}}$ be an i.i.d. family of random
        variables with values in the 2-torus $\T^{2}$. We write
        $\theta(n)=\left(\theta_{1}(n),\theta_{2}(n)\right)$. We define a family
        of complex random variables
        $(Y(n))_{n\in\N^{2}}$ by
        $$
        Y(n_{1},n_{2}):=
  e\left(r_{1}\theta_{1}(m_{1},m_{2})+r_{2}\theta_{2}(m_{1},m_{2})\right)\;,
        $$
        if $n_{i}=m_{i}^{2}+r_{i}$, with $0\leq r_{i}\leq2m_{i}$, $i=1,2$.

        We have, almost surely,
        $$
        \lim_{N_{1},N_{2}\to+\infty}\frac1{N_{1}N_{2}}\sum_{0<
        n\leq N}Y(n) =\P\left(\theta=0\right)\;.
        $$
        \end{lemma}

\begin{lemma}\label{lemmeprob.2}
Let $(X(n))_{n\in\N^{2}}$ be an i.i.d. family of bounded complex
random variables. We define a new family of complex random variables
$(Z(n))_{n\in\N^{2}}$ by
$$
        Z(n_{1},n_{2}):=X(m_{1},m_{2})
        $$
        if $n_{i}=m_{i}^{2}+r_{i}$, with $0\leq r_{i}\leq2m_{i}$, $i=1,2$.

        We have, almost surely,
         $$
        \lim_{N_{1},N_{2}\to+\infty}\frac1{N_{1}N_{2}}\sum_{0<
        n\leq N}Z(n) =\E\left[X\right]\;.
        $$
        \end{lemma}
Let us explain briefly how (S1)$\Rightarrow$(S2) follows from these lemmas.
        
Suppose that a vdC set $D\subset\Z^{2}$ and
a measure $\sigma$ on $\T^{2}$ are given.
We suppose that the Fourier transform of $\sigma$ is null on $D$.
Without loss of generality, we can suppose that $\sigma$ is a
probability measure, and we consider a family of random
variables
$(\theta(n))_{n\in\N^{2}}$ independent and of law $\sigma$. We
define, as in Lemma \ref{lemmeprob.1}, a family of complex random
variables $(Y(n))_{n\in\N^{2}}$. A slight modification\footnote{Details are provided after Lemma \ref{lemmeprob.2.1} in Section \ref{enhanced-def-sec} .} of
Lemma~\ref{lemmeprob.2} gives us the following result: for all
$h\in\Z^{2}$, almost surely,
$$
\lim_{N_{1},N_{2}\to+\infty}\frac1{N_{1}N_{2}}\sum_{0<n\leq N}
Y(n+h)\overline{Y(n)}=\E\left[e(h_{1}\theta_{1}+h_{2}\theta_{2})\right]\;.
$$
This last quantity is exactly $\widehat{\sigma}(h)$ and, by
hypothesis, it is null for $h\in D$. Since $D$ is a vdC set,
we conclude that
$$
\lim_{N_{1},N_{2}\to+\infty}\frac1{N_{1}N_{2}}\sum_{0<n\leq N}
Y(n)=0\;.
$$
By Lemma \ref{lemmeprob.1}, this means that
$\P\left(\theta=0\right)=0$, i.e. 
$\sigma\left(\{(0,0)\}\right)=0$. 
\end{proof}

\subsubsection{Some corollaries}\label{ex.}
Here are some direct applications of the spectral characterization.
\begin{corollary}[Ramsey property. Cf. \cite{R}, Corollary 1.]\label{Ramsey}
        If $D=D_{1}\cup D_{2}$ is a vdC set in $\Z^{2}$, then at least one
        of the sets $D_{1}$ or $D_{2}$ is a vdC set. (In particular,
        if $D$ is a vdC set in $\Z^{2}$ and $E$ is a finite subset of $D$,
        then $D\setminus E$ is still a vdC set in $\Z^{2}$.)
        \end{corollary}
\begin{proof}
        If $\sigma_{1}$ and $\sigma_{2}$ are positive measures on $\T^{2}$
such that $\widehat{\sigma_{i}}$ is null on $D_{i}$, then the Fourier transform of their
convolution  $\sigma_{1}\star\sigma_{2}$ vanishes on $D_{1}\cup D_{2}$. And  $\sigma_1\star\sigma_2(\{0\})\geq\sigma_1(\{0\})\times\sigma_2(\{0\})$. \end{proof}
If $\F$ is a family of subsets of $\Z^2$, we denote by $\F^*$ its dual family, that is the family of all sets $G\subset \Z^2$ such that $G\cap F\neq\emptyset$ for all $F\in\F$. The Ramsey property described in Corollary \ref{Ramsey} has a remarkable consequence for the family of vdC$^*$ sets : if $A$ is a vdC set and if $B$ is a vdC$^*$ set, then $A\cap B$ is a vdC set ; this impies that the family of vdC$^*$ sets is stable with respect to finite intersections, hence is a filter.

\begin{corollary}[Sets of differences]
        If $I$ is an infinite subset of $\Z^{2}$, then the set of
        differences $D:=\{n-m\,:\,n,m\in I\ \text{and}\ n\neq m\}$ is a vdC set.
        \end{corollary}
        \begin{proof}
	Suppose that $\sigma$ is a
probability measure on $\T^{2}$, whose Fourier transform vanishes on $D$.
This means that the characters $x\mapsto e(n\cdot x)$, with $n\in I$, form
an orthonormal family in $L^{2}(\sigma)$. For any finite subset $J$
of $I$, we have
$$
(\mbox{card}J)^{2}\sigma(\{(0,0)\})\leq\int_{\T^{2}}\left|\sum_{n\in
J}e(n\cdot x)\right|^{2}\ \mbox{d}\sigma(x)=\mbox{card}J\;.
$$
This implies that $\sigma$ has no point mass at zero.
\end{proof}
\begin{remark}The preceding proof gives, in fact, more. Namely, \emph{any set $D$ which 
contains sets of differences of arbitrarily large finite sets is a vdC set.}
\end{remark}
\begin{corollary}[Linear transformations of vdC sets]\label{linear-transf}
Let $d$ and $e$ be positive integers, and let $L$ be a linear
transformation from $\Z^d$ into $\Z^e$ (i.e. an $e\times d$ matrix with
integers entries).
\begin{enumerate}
        \item
        If $D$ is a vdC set in $\Z^d$ and if $0\notin L(D)$, then $L(D)$ is a vdC set in $\Z^e$.
        \item
        Let $D\subset\Z^d$. If the linear map $L$ is one to one, and if $L(D)$ is a vdC set
        in $\Z^e$, then $D$ is a vdC set in $\Z^d$.
\end{enumerate}
\end{corollary}
\begin{proof}
Let $D$ be vdC set in $\Z^d$ and $\sigma$ a positive measure on the
$e$-torus such that $\widehat{\sigma}$ vanishes on $L(D)$. Let us denote by $^t\!L$ the map from $\T^e$ into $\T^d$ defined by $k\cdot^t\!\!L(x)=L(k)\cdot x$ for $k\in\Z^d$ and $x\in\T^e$. Denoting by $\sigma'$ the image of $\sigma$ under the linear transformation $^t\!L$,
     we see that, for all $k\in\Z^d$, $\widehat{\sigma'}(k)=\widehat\sigma(L(k))$. Hence the Fourier
transform $\widehat{\sigma'}$ vanishes on the vdC set $D$. The measure $\sigma'$
has no mass at zero, and hence $\sigma$ also has no mass at zero. This proves the
first assertion.

Suppose now that $L$ is one to one and that $L(D)$ is a vdC set in $\Z^e$.
Consider the lattice $L(\Z^d)$ in $\Z^e$.
By a classical lemma (see for example \cite{godement}, Exercise 8 of Chapter 31), there
exist $n_1, n_2,\ldots, n_e$ in $\Z^e$ and positive
integers $p_1,p_2,\ldots,p_d$ such that $\Z^e=\Z n_1+\Z n_2+\ldots+\Z n_e$ and $L(\Z^d)= p_1\Z n_1+p_2\Z n_2+\ldots+p_d\Z n_d$. This allows us to view $L(D)$ as a vdC set in $\Z^d\simeq \Z n_1+\Z n_2+\ldots+\Z n_d$ and $L$ as an endomorphism of $\Z^d$.

Let $\sigma'$ be a positive measure on the $d$-torus such that
$\widehat{\sigma'}$ vanishes on $D$. The linear map $^t\!L$ from
$\T^d$ into $\T^d$ is finite to one and
onto. Since it is onto, it posseses an inverse on the right and we can
see $\sigma'$ as
the image of a positive measure $\sigma$ on the $d$-torus, under the
map $^t\!L$. The Fourier transform $\widehat{\sigma}$ vanishes on
$L(D)$, hence the measure $\sigma$ is continuous. Since the map $^t\!L$
is finite to one, we conclude that the measure $\sigma'$ is also
continuous. This proves the second assertion.
\end{proof}
\begin{corollary}[Lattices are vdC*]\label{lattices}
        If $G$ is any $d$-dimensional lattice in $\Z^d$, and if $D$ is a vdC
        set in $\Z^d$, then $G\cap D$ is a vdC set in $\Z^d$.
        \end{corollary}
\begin{proof}
To begin, we remark that if $G$ is a lattice in $\Z^d$, and if
$z\in\Z^d$, $z\notin G$, then the translate $z+G$ is not a vdC set
in $\Z^d$ (test the definition of a vdC set on the indicator function
of the set $G$). Since $G$ is a $d$-dimensional lattice in $\Z^d$,
there exist finitely many points $z_{1},z_{2},\ldots,z_{k}$ in $\Z^d$
and outside $G$ such that
$$
\Z^d=G\cup\big(\cup_{i=1}^k(z_{i}+G)\big)\;.
$$
Let $D$ be a vdC set in $\Z^d$. We have
$$
D=(G\cap D)\cup\big(\cup_{i=1}^k(z_{i}+G)\cap D\big)\;.
$$
Since none of the sets $(z_{i}+G)\cap D$ is vdC, Corollary
\ref{Ramsey} tells us that the set $(G\cap D)$ is vdC.
\end{proof}
\begin{remark}As a consequence of the last two statements, we note the following fact, which is the direct extension of Corollary 2 in \cite{R}. 

\emph{Let $L$ be a one to one linear transformation from $\Z^d$ into itself; let $D$ be a vdC set in $\Z^d$; the set of $n\in\Z^d$ such that $L(n)\in D$ is a vdC set in $\Z^d$.}

Indeed, by Corollary \ref{lattices}, $D\cap L(\Z^d)$ is a vdC set in $\Z^d$ and, by Corollary~\ref{linear-transf}, its inverse image by $L$ is a vdC set.\\
\end{remark}

The spectral characterization also implies that various formulations of the vdC property, 
associated to different averaging methods, are in fact equivalent (see Section 
\ref{dist-notion}).
\subsection{The Kamae - Mend\`es France criterion}
\subsubsection{The criterion}
Let $D\subset\Z^{2}$ and let $P$ be a real trigonometric polynomial on
$\T^{2}$. We say that {\em{the spectrum of $P$ is contained in $D$}} if
$P$ is a linear combination of the characters $(x_{1},x_{2})\mapsto
e(d_{1}x_{1}+d_{2}x_{2})$ with $(d_{1},d_{2})\in\pm D$. In the case of a
one dimensional space of parameters, the following proposition appears in Ruzsa's article \cite{R}, with the same proof.

\begin{proposition}\label{K-MF1}
A subset $D$ of $\Z^{2}\setminus\{0\}$ is a van der Corput set if and only if
for all $\epsilon>0$, there
exists a real trigonometric polynomial $P$ on the 2-torus $\T^{2}$
whose spectrum is contained in $D$ and which satisfies $P(0)=1$,
$P\geq-\epsilon$.
\end{proposition}
\begin{proof}
Let us assume that there exists such a trigonometric polynomial.
Let $\sigma$ be a positive measure on $\T^{2}$ whose Fourier
transform $\widehat\sigma$ is null on $D$. Then we have
$$
\int_{\T^{2}}P\ \mbox{d}\sigma=0\;.
$$
But from $P(0)=1$ and $P\geq-\epsilon$ we deduce that
$$
\int_{\T^{2}}P\
\mbox{d}\sigma\geq\sigma\left(\{0\}\right)-
\epsilon\sigma\left(\T^2\setminus\{0\}\right)\;.
$$
Thus we necessarily have $\sigma\left(\{0\}\right)=0$, and we deduce from
Theorem \ref{spectralchar} that $D$ is a vdC set.

For the proof of the inverse implication, we follow Ruzsa's argument (\cite{R}, Section 5).
We will denote $m\cdot x:=m_{1}x_{1}+m_{2}x_{2}$ if
$x=(x_{1},x_{2})\in\T^{2}$
and $m=(m_{1},m_{2})\in\Z^{2}$.

Let us suppose that $D$ is a subset of $\Z^{2}$ and that there exists
$0<\epsilon<1$ such that, for any real trigonometric polynomial $P$
with spectrum in $D$ and such that $P(0)=1$, we have
$\min(P+\epsilon)\leq0$. In the Banach space $\C_\R\!\left(\T^2\right)$ of real continuous
functions on $\T^{2}$, equipped with the uniform norm, we consider the
set
$\mathcal F$ of strictly positive functions and the set
$\mathcal Q$ of real trigonometric polynomials $P$, with spectrum in $D$
and such that $P(0)=1$. By hypothesis, the convex sets
$\mathcal F$ and $\epsilon+\mathcal
Q$ are disjoint. By Hahn-Banach Theorem, there exists a non-zero real-valued continuous linear functional $L$ on $\C_\R\!\left(\T^2\right)$, which takes nonnegative values on $\mathcal F$ and nonpositive values on $\epsilon+\mathcal Q$. Let us denote by $\sigma$ the measure on $\T^2$ associated to $L$ by Riesz representation theorem : $L(f)=\int_{\T^2}f\,\text{d}\sigma$, for all $f\in\C_\R\!\left(\T^2\right)$. Since $L\geq0$ on $\mathcal F$, this measure is positive and we
can assume that it is normalized. Let $m,n\in\pm D$. If $P\in\mathcal Q$,
then, for all $\lambda\in\R$, the function $x\mapsto
\epsilon+P+\lambda(\cos 2\pi (m\cdot x)-\cos 2\pi (n\cdot x))$
is still in $\epsilon+\mathcal
Q$. This implies that $\int\cos 2\pi (m\cdot x)\ \mbox{d}\sigma(x)=
\int\cos 2\pi (n\cdot x)\ \mbox{d}\sigma(x)$. Similarly, for all
$\lambda\in\R$,
the function $x\mapsto
\epsilon+P+\lambda\sin 2\pi (m\cdot x)$ is still in $\epsilon+\mathcal
Q$, and this implies that $\int\sin 2\pi (m\cdot x)\ \mbox{d}\sigma(x)=0$.

We define $r:=\int\cos 2\pi (m\cdot x)\ \mbox{d}\sigma(x)$, for $m\in\pm D$. If
$P\in\mathcal Q$, we have $$\int_{\T^{2}}(\epsilon+P)\
\mbox{d}\sigma\leq 0$$ and, writing $$P(x)=\sum_{m\in\pm D} a_m \cos2\pi(m\cdot x) + b_m\sin2\pi(m\cdot x)\;,$$ we have $$\int_{\T^{2}}P\
\mbox{d}\sigma = r\sum_{m\in\pm D} a_m=rP(0)=r\;.$$
Hence $r\leq-\epsilon<0$. Denoting by $\delta$ the Dirac mass at
$0$, we consider a new probability measure $\sigma'$ defined by
$$
\sigma':=\frac1{1-r}(\sigma-r\delta)\;.
$$
We have $\sigma'(\{0\})\geq\frac{-r}{1-r}>0$.

But this probability satisfies $\widehat{\sigma'}(m)=0$ for all $m\in
D$, and, using Theorem~\ref{spectralchar}, we conclude that
$D$ is not a vdC set.
\end{proof}


\subsubsection{Application to polynomial sequences and sequences of shifted primes}\label{pol+prime}
The following proposition is the two-dimensional extension of Example 3 in \cite{K-MF}.
\begin{proposition}\label{K-MF3}
        Let $D\subset\Z^{2}$. For each $q\in\N$, we denote
        $$D_{q}:=\left\{(d_{1},d_{2})\in D\,:\,q! \mbox{ divides
        }d_{1}\mbox{ and }d_{2}\right\}\;.$$
        Suppose that, for every $q$,
        there exists a
        sequence $(h^{q,n})_{n\in\N}$ in $D_{q}$ such that, for every
        $x=(x_{1},x_{2})\in\R^{2}$, if $x_{1}$ or $x_{2}$ is irrational,
        the sequence $\left(h^{q,n}\cdot x\right)_{n\in\N}$ is
        uniformly distributed mod 1.
        Then $D$ is a vdC set.
\end{proposition}
\begin{proof}
Let us define a family of trigonometric polynomials with spectrum
contained in $D$, by the formula
\begin{equation}\label{trigo-mean}
P_{q,N}(x):=\frac1N\sum_{n=1}^{N}e\left(h^{q,n}\cdot x\right)\;,
\end{equation}
where $q$ and $N$ are positive integers and $x\in\R^{2}$.
By hypothesis, if $x\notin\Q^{2}$ then
$\lim_{N\to+\infty}P_{q,N}(x)=0$. For each $q$, there exists a subsequence
$\left(P_{q,N'}\right)$ which is pointwise convergent to a function
$g_{q}$. For all $x\in\Q^{2}$, we have $g_{q}(x)=1$ for all large
enough $q$, and for all  $x\notin\Q^{2}$, we have $g_{q}(x)=0$. The
sequence $(g_{q})$ is pointwise convergent to the characteristic
function of $\Q^{2}$. 
Consider now a positive measure $\sigma$ on $\T^2$ whose Fourier transform $\widehat\sigma$ vanishes on $D$. We have $\int P_{q,N}\,d\sigma=0$ for all $q,N$. Applying the dominated convergence theorem twice, we conclude that $\sigma(\Q^2)=0$. In particular $\sigma(\{0\})=0$, and we are done.
\end{proof}
A sequence $(d_n)_{n\in\N}$ in $\Z^2$ will be called a \emph{vdC sequence} if the set of its values $\{d_n\,:\,n\in\N\}$ is a vdC set.

The ($d$-dimensional version of the) following proposition extends Theorem~4.2 in \cite{bergelson}.
\begin{proposition}\label{polynomials}
Let $p_{1}$ and $p_{2}$ be two polynomials with integer coefficients.
The sequence $(p_{1}(n),p_{2}(n))_{n\in\N}$
is a vdC sequence in $\Z^{2}$ if and only if for all positive integers $q$, there
exists $n\geq1$ such that $q$ divides $p_{1}(n)$ and $p_{2}(n)$.
\end{proposition}

Note that the divisibility condition is satisfied if
$p_{1}$ and $p_{2}$ have zero constant term.

\begin{proof}[Proof of Proposition \ref{polynomials}]
By Corollary \ref{lattices}, the divisibility condition is necessary for the sequence 
$(p_{1}(n),p_{2}(n))$ to be vdC. Let us prove that this condition is sufficient. We are 
going to distinguish two cases: either $p_{1}$
and $p_{2}$ are proportional, or not.

In the first case, there exists a
polynomial $p\in\Z[X]$, and integers $a$, $b$ such that $p_{1}=ap$ and
$p_{2}=bp$. The polynomial $p$ satisfies the divisibility property,
which ensures that $(p(n))$ is a vdC sequence in $\Z$ (it is a direct consequence of the one-dimensional version of Proposition \ref{K-MF3}, cf.
\cite{K-MF}). By the first statement of Corollary \ref{linear-transf}, this implies
that $\left(ap(n),bp(n)\right)$ is a vdC sequence in $\Z^{2}$.

Consider now the second case, in which polynomials $p_{1}$
and $p_{2}$ are not proportional.
Let $q$ be a positive integer and
$(x_{1},x_{2})\in\R^{2}\setminus\Q^{2}$ ; there exists $n\geq1$ such that
$q!|p_{1}(n)$ and $q!|p_{2}(n)$ ; for all $k\in\Z$, we have
$q!|p_{1}(n+kq!)$ and $q!|p_{2}(n+kq!)$. We claim that the sequence
\begin{equation}\label{comb}
\left(p_{1}(n+kq!)x_{1}+p_{2}(n+kq!)x_{2}\right)_{k\in\N}
\end{equation}
is uniformly distributed mod 1. This fact implies, by Proposition \ref{K-MF3}, that
$(p_{1}(n),p_{2}(n))$ is a vdC sequence in $\Z^{2}$.
In order to prove the claim, we consider first the case when 1, $x_1$ and $x_2$ 
are linearly independent over $\Q$ ~; in this case the sequence (\ref{comb}) is u.d. mod 1 
by Weyl's theorem. Let us consider now the case in which 1, $x_1$ and $x_2$ are linearly dependent 
over $\Q$ and $x_1$ is irrational ; in this case we have $x_2=rx_1+s$, with $r,s\in\Q$, and, if $q$ has been chosen enough large, the sequence (\ref{comb}) has (mod 1) the form 
$$
\left((p_{1}(n+kq!)+rp_{2}(n+kq!))x_{1}\right)_{k\in\N}\;;
$$
we conclude once more by Weyl's theorem since the polynomial $p_1+rp_2$ is not constant. Finally, if $x_1$ is rational, then $x_2$ is irrational and the argument is similar.
\end{proof}
\begin{remark} (See Appendix) There exist pairs of polynomials $p_1, p_2$ satisfying:

 - for all integers $a$ and $b$ and for all positive integers $q$, there exists $n$ 
such that $q\mid ap_1(n)+bp_2(n)$ (hence $(ap_1(n)+bp_2(n))_{n\in\N}$ is a vdC sequence in $\Z$).
 
 - there exists a positive integer $q$ such that for no $n$ are the numbers $p_1(n)$ and 
$p_2(n)$ simultaneously multiples of $q$ (hence $(p_1(n),p_2(n))_{n\in\N}$ is not a vdC sequence in $\Z^2$). \\
 \end{remark}

Let $P$ be the set of prime numbers. It is shown in \cite{K-MF} that $P-~1$ and $P+1$ 
are vdC sets, and that no other translate of $P$ is a vdC set. This can be extended to 
polynomials along $P-1$ and $P+1$, and to the multidimensional setting. For example, 
we have the following result.
\begin{proposition}\label{pp}
Let $f,g$ be two (non zero) polynomials with integer coefficients and zero constant term. 
The set $\{(f(p-1),g(p-1))\,:\, p\in P\}$ is a vdC set in $\Z^2$.
\end{proposition}

The proof of this proposition relies on Proposition \ref{K-MF3} and on the following Vinogradov-type theorem. 
\begin{theorem}\label{vino-rhin}
Let $q$ be a positive integer and $h$ be a real polynomial such that the polynomial 
$ h- h(0)$ has at least one irrational coefficient. The sequence $(h(p))$ is uniformly 
distributed mod 1, where $p$ describes the increasing sequence of prime numbers in the 
congruence class $1+q\N$.
\end{theorem}

The proof of this theorem can be given in a few sentences, by ``quotation''. It is proved in \cite{rhin} (see also \cite{nieder.}) that if a real polynomial $\tilde h$ is such that $\tilde h-\tilde h(0)$ has at least one irrational coefficient, then\begin{equation}\label{vino}\text{the sequence $(\tilde h(p))_{p\in P}$ is u.d. mod 1.}\end{equation} 

Now we can use the following simple trick (cf. \cite{Mont}, p.34) :
$$
\sum_{\substack{p\leq n\\p\equiv1[q]}} e(h(p))=\frac1q\sum_{j=1}^qe\left(-\frac{j}{q}\right)\sum_{p\leq n} e(h(p)+pj/q)\;.
$$
After division by $\pi(n)$, the right term goes to zero as $n$ goes to infinity because (\ref{vino}) can be applied to $\tilde h(p)=h(p)+pj/q$. 

Moreover it is well known that the Prime Number Theorem has a natural extension to the 
distribution of primes in arithmetic progressions : the number of primes less than $n$ 
in $1+q\N$ is asymptotically equivalent to $\pi(n)/\varphi(q)$ as $n$ goes to infinity. 

We obtain that
$$
\lim_{n\to+\infty}\frac1{\#\{p\leq n,\,p\equiv1[q]\}}\sum_{\substack{p\leq n\\p\equiv1[q]}}  
e(h(p))=0\;.
$$
This is still true when we replace $h$ by a non zero integer multiple of $h$, which, 
via Weyl's criterion, gives uniform distribution (mod 1) of the sequence 
$\left((h(p))_{p\in P,\,p\equiv1[q]}\right)$.

\begin{proof}[Proof of Proposition \ref{pp}] This proof is parallel to the proof of 
Proposition \ref{polynomials}. If $f$ and $g$ are proportional, we use the 
fact that $(f(p-1))_{p\in P}$ is a vdC sequence (which is a direct consequence of 
the one-dimensional version of Proposition \ref{K-MF3} and of Theorem \ref{vino-rhin}). 
If $f$ and $g$ are not proportional, we deduce from Theorem \ref{vino-rhin} that for 
all large enough positive integers $q$, and for all $(x_1,x_2)\in\R^2\setminus\Q^2$, 
the sequence 
$
\left(f(p-1)x_1+g(p-1)x_2\right)_{p\in P,\,p\equiv1[q]}
$ is u.d. mod 1. We conclude by Proposition \ref{K-MF3}.
\end{proof}
 Several other examples of vdC sets are presented in Subsection \ref{ex-enh-vdc}.
\subsubsection{One more corollary \`a la Ruzsa}

Following \cite{R}, we deduce from Proposition~\ref{K-MF1} a new combinatorial property of vdC sets.
\begin{corollary}[Cf. \cite{R}, Corollary 3]\label{last}
Any vdC set in $\Z^{2}$ can be partitionned into infinitely many
pairwise disjoint vdC sets.
\end{corollary}

\begin{proof} Let $D$ be a vdC set in $\Z^{2}$.
There exists a sequence $(I_{k})_{k\geq1}$ of pairwise disjoint
finite subsets of $D$, and for each $k$, a trigonometric polynomial
$P_{k}$ with spectrum in $I_{k}$ and such that $P_{k}(0)=1,
P_{k}+\frac1k>0$. The existence of $I_{k}$ and $P_{k}$ can be proved
by induction using the direct implication in Proposition \ref{K-MF1}
and the fact that, for each $k$, the set $D\setminus\left(I_{1}\cup
I_{2}\cup\ldots\cup I_{k}\right)$ is vdC (see Corollary \ref{Ramsey}).
From
the inverse implication in Proposition \ref{K-MF1}, we deduce that
any infinite union of the $I_{k}$'s is a vdC set. We can consider an
infinite family of pairwise disjoint such sets.
\end{proof}

\subsection{Positive-definite multiparameter sequences and generalized vdC
inequality}\label{inequality}
\subsubsection{The inequality}

We show in this subsection that the Kamae-Mend\`es France criterion can
be formulated in terms of positive-definite sequences. This will allow us, for a given 
vdC set $D$, to obtain a quantitative van
der Corput type inequality in which only correlations $\gamma(N,d)$ for $d\in D$ are involved. 

\begin{proposition}\label{pos-def}
        Let $(a_{h})_{h\in\Z^{2}}$ be a family of complex
numbers such that all but finitely many of $a_h$ are zero. This family is positive-definite
if and only if the trigonometric polynomial $T(x):=\sum_{h}a_{h}e(h\cdot x)$,
$x\in\R^{2}$, takes only nonnegative values.
\end{proposition}

\begin{proof}
        Recall that the family $(a_{h})$ of complex numbers
        is positive-definite if, for any
        family $(z_{h})_{h\in\Z^{2}}$ of complex numbers, all zero but finitely
        many,
        $$
        \sum_{h,h'\in\Z^{2}}a_{h-h'}z_{h}\overline{z_{h'}}\geq0\;.
        $$
        We will denote $h=(h_{1},h_{2})$.

        The family $(a_{h})$ is the Fourier transform of the measure having
        density $T$ with respect to Lebesgue measure on the 2-torus.
        Thus it is clear that if the trigonometric polynomial is
        positive, then the family is positive-definite.
In the opposite direction, suppose that $(a_{h})$ is positive-definite (and that
        $a_{h}=0$ for all $h$ but finitely many). For $x\in\mathbb R^2$ and for all positive integers $c$,
     $$
        \sum_{0\leq h,h'<(c,c)}a_{h-h'}e(h\cdot x)e(-h'\cdot x)\geq0\;.
        $$
This can be written
$$
\sum_{(-c,-c)<h<(c,c)}(c-|h_{1}|)(c-|h_{2}|)a_{h}e(h\cdot x)\geq0\;.
$$
Dividing this expression by $c^{2}$, and letting $c$ go to infinity,
we obtain
$$
\sum_{h}a_{h}e(h\cdot x)\geq0\;.
$$
\end{proof}
\begin{remark}\label{pos-def-crit}
        The Kamae - Mend\`es France criterion (Proposition \ref{K-MF1})
can now be rewritten as follows : a subset $D$ of $\Z^2\setminus\{0\}$ is a
vdC set if and only if, for all $\epsilon>0$, there exists a positive-definite family $(a_{d})_{d\in\Z^2}$ such that :
\begin{itemize}\item
all but finitely many $a_{d}$ are zero;
\item
$a_{d}=0$ whenever $d\neq 0$ and $d\notin D\cup(-D)$;
\item
$a_{0}\leq\epsilon$ and $\sum_{d}a_{d}=1$.
\end{itemize}\end{remark}
\medbreak
As in the first section, we will denote
$$
\gamma(N,h):=\sum_{\begin{subarray}{1}0< n\leq N\\0< n+h\leq N\end{subarray}}
u_{n+h}\cdot\overline{u_{n}}\;,
$$
if $h\in\Z^2$, $N\in\N^{2}$ and $(u_{n})_{0\leq n<N}$ is a family of
complex numbers. We will write also
$$
\|u\|_{\infty}:=\max_{n}|u_{n}|\;.
$$
\begin{theorem} \label{gen.vdC.ineq}
        Let $H\in\N^{2}$ and $(a_{h})_{-H<h<H}$ be a finite
positive-definite family of complex numbers, with $\sum_{h}a_{h}=1$.
Let $N\in\N^{2}$ and $(u_{n})_{0<n\leq N}$ be a finite family of
complex numbers. We have
$$
\Big|\sum_{0<n\leq N}u_{n}\Big|^{2}\leq N_{1}N_{2} \Big(
\sum_{h}a_{h}\gamma(N,h)
+5\|u\|_{\infty}^{2}\sum_{h}(|h_{1}|N_{2}+|h_{2}|N_{1}
+|h_{1}h_{2}|)|a_{h}|\Big).
$$
\end{theorem}
This inequality should be compared to the ``generalized van der Corput
Lemma" stated in \cite{Mont} (Chap.2, Lemma 1).

If we consider a bounded family of complex numbers $(u_{n})_{n\in\N^2}$,
we deduce from Theorem \ref{gen.vdC.ineq} the following inequality
$$
\left|\frac1{N_{1}N_{2}}\sum_{0<n\leq N}u_{n}\right|^{2}\leq
\sum_{h}a_{h}\frac1{N_{1}N_{2}}\gamma(N,h)+
O\left(\max\left(\frac1{N_{1}},\frac1{N_{2}}\right)\right)\;,
$$
which will be utilized when describing the vdC property of Cartesian products of vdC sets.

Corollary \ref{quant} below, which is a direct consequence of Theorem 
\ref{gen.vdC.ineq}, gives what one might call a quantitative version of the 
van der Corput trick. The ``if'' part of the Kamae-Mend\`es France criterion is a direct
consequence of this corollary.

\begin{corollary}\label{quant}
         Let $(a_{h})_{-H<h<H}$ be a
positive-definite family of complex numbers, and $(u_{n})_{n\in\N^{2}}$
be a family of complex numbers. If, for any $h$ such that $h\neq0$ and
$a_{h}\neq0$ we have
$$
\lim_{N_{1},N_{2}\to+\infty}\frac1{N_{1}N_{2}}\gamma(N,h)=0
$$
then
$$
\limsup_{N_{1},N_{2}\to+\infty}\left|\frac1{N_{1}N_{2}}\sum_{0<
n\leq N}u_{n}\right|\leq \|u\|_{\infty}\sqrt{a_{0}}\;.
$$
\end{corollary}

\begin{proof}[Proof of Theorem \ref{gen.vdC.ineq}]
Let us define
$$
m:=\frac1{N_{1}N_{2}}\sum_{0< n\leq N}u_{n}\quad\text{and}\quad
v_{n}:=u_{n}-m\;.
$$
We have
$$
\gamma(N,h)=\sum_{\begin{subarray}{1}0<n\leq N\\0<n+h\leq N\end{subarray}}
(v_{n+h}+m)(\overline{v_{n}}+\overline{m})=A_{h}+B_{h}+C_{h}+D_{h}\;,
        $$
where :
$$
A_{h}:=\sum_{\begin{subarray}{1}0<n\leq N\\0<n+h\leq N\end{subarray}}
v_{n+h}\cdot\overline{v_{n}}\;,
\qquad
B_{h}:=m\sum_{\begin{subarray}{1}0<n\leq N\\0<n+h\leq N\end{subarray}}
 \overline{v_{n}}\;,
$$
$$
C_{h}:=\overline{m}\sum_{\begin{subarray}{1}0<n\leq N\\0<n+h\leq N\end{subarray}}
 v_{n+h}\;,
\qquad
D_{h}:=|m|^{2}\sum_{\begin{subarray}{1}0<n\leq N\\0<n+h\leq N\end{subarray}}
1\;.
$$
Since the family $(a_{h})$ is positive-definite, we have
$$
\sum_{h}a_{h}A_{h}\geq0\;.
$$
The number of points $n$ in the square $[1,N_{1}]\times[1,N_{2}]$, such
that we do not have $0<n+h\leq N$, is less or equal than
$|h_{1}|N_{2}+|h_{2}|N_{1}$. Since
$\displaystyle
\sum_{0<n\leq N}v_{n}=0
$
we deduce that
\begin{multline*}
|B_{h}|\leq|m|\left(|h_{1}|N_{2}+|h_{2}|N_{1}\right)\|v\|_{\infty}\leq
2|m|\left(|h_{1}|N_{2}+|h_{2}|N_{1}\right)\|u\|_{\infty} \\ \leq
2\left(|h_{1}|N_{2}+|h_{2}|N_{1}\right)\|u\|_{\infty}^{2}\;.
\end{multline*}
The same inequality holds for $|C_{h}|$.\\
We have also \begin{multline*}
D_{h}=(N_{1}-|h_{1}|)(N_{2}-|h_{2}|)|m|^{2}\\ \geq
\frac1{N_{1}N_{2}}\Big|\sum_{0< n\leq N}u_{n}\Big|^2 -
\left(|h_{1}|N_{2}+|h_{2}|N_{1} +
|h_{1}h_{2}|\right)\|u\|_{\infty}^{2}\;.
\end{multline*}
    From these inequalities, we deduce that
\begin{multline*}
\sum_{h}a_{h}\gamma(N,h)\geq \sum_{h}a_{h}A_{h} + \sum_{h}a_{h}D_{h}
-\sum_{h}a_{h}\left(|B_{h}|+|C_{h}|\right) \\ \geq
\frac1{N_{1}N_{2}}\Big|\sum_{0<n\leq N}u_{n}\Big|^2-5\sum_{h}|a_{h}|
\left(|h_{1}|N_{2}+|h_{2}|N_{1}
+|h_{1}h_{2}|\right)\|u\|_{\infty}^{2}\;.
\end{multline*}
and the result follows.
\end{proof}

In the next two subsections we present corollaries of Theorem \ref{gen.vdC.ineq}.

\subsubsection{Cartesian products of vdC sets}
\begin{corollary}\label{cartesian}
Let $k$, $\ell$ be positive integers, and $D$, $E$ be vdC
sets in, respectively, $\Z^k$ and $\Z^\ell$. The product set $D\times E$
is a vdC set in $\Z^{k+\ell}$.
\end{corollary}

\begin{proof}

Let us consider, as a typical example, the case $k=\ell=2$. We consider
two vdC sets $D$ and $E$ in $\Z^{2}$. Let
$(u_{n,m})_{n,m\in\Z^2}$ be a family of complex numbers of modulus one
indexed by $\Z^4$, and satisfying : for all $d\in D$ and all $e\in E$,
\begin{equation}\label{H}
\lim_{\begin{subarray}{1}N_{1},N_{2}\to+\infty\\M_{1},M_{2}\to+\infty
\end{subarray}}\frac1{N_{1}N_{2}M_{1}M_{2}}
\sum_{\begin{subarray}{1}0\leq n<(N_{1},N_{2})\\0\leq m<(M_{1},M_{2})
\end{subarray}}u_{n+d,m+e}\cdot\overline{u_{n,m}}=0\;.
\end{equation}
It is not hard to verify that (\ref{H}) is still true when $d\in(-D)$ or $e\in(-E)$.

Let us fix $\epsilon>0$. By Remark \ref{pos-def-crit}, there exist two
positive-definite families $(a_{d})$ and $(b_{e})$ indexed by $\Z^2$
such that $a_{d}$ (resp. $b_{e}$) is zero whenever $d$ (resp. $e$)
is outside a finite subset of $D\cup(-D)\cup\{0\}$ (resp. $E\cup(-E)\cup\{0\}$),
with $a_{0}<\epsilon$, $b_{0}<\epsilon$ and
$\sum_{d}a_{d}=\sum_{e}b_{e}=1$.

It is clear from Proposition \ref{pos-def} (or from the Bochner-Herglotz
Theorem) that the family $(a_{d}b_{e})_{(d,e)\in\Z^{4}}$ is positive
definite. Let us denote $P:=N_{1}N_{2}M_{1}M_{2}$ and
$p:=\min\{N_{1},N_{2},M_{1},M_{2}\}$.
The generalized vdC inequality (Theorem~\ref{gen.vdC.ineq})
applied to $\Z^{4}$ gives
$$
\left|\sum_{\begin{subarray}{1}0\leq n<(N_{1},N_{2})\\0
\leq m<(M_{1},M_{2})\end{subarray}}u_{n,m}\right|^{2}\leq
P\sum_{d,e}a_{d}b_{e}
\sum_{\begin{subarray}{1}0\leq n,n+d<(N_{1},N_{2})\\0\leq m,m+e<(M_{1},M_{2})
\end{subarray}}u_{n+d,m+e}\cdot\overline{u_{n,m}} + P^{2}
O\left(1/p\right)\;.
$$
Dividing by $P^{2}$, letting $p$ go to infinity and using (\ref{H}),
we obtain
$$
\limsup_{\begin{subarray}{1}N_{1},N_{2}\to+\infty\\M_{1},M_{2}\to+\infty
\end{subarray}}\left|\frac1{N_{1}N_{2}M_{1}M_{2}}
\sum_{\begin{subarray}{1}0\leq n<(N_{1},N_{2})\\0
\leq
m<(M_{1},M_{2})\end{subarray}}u_{n,m}\right|^{2}\leq
\sum_{d\text{ or }e=0}a_{d}b_{e}\;.
$$
Since $\displaystyle
\sum_{d\text{ or }e=0}
a_{d}b_{e}=a_{0}\sum_{e}b_{e}+b_{0}\sum_{d}a_{d}-a_{0}b_{0}\leq2\epsilon$,
we conclude that the last limsup is zero.
\end{proof}

\subsubsection{Sequences in Hilbert space}
The goal of this short subsection is to point out that generalized van der Corput 
inequalities can be extended from numerical sequence to sequences of vectors in a 
Hilbert space. One of the reasons to be interested in such extensions is that they 
provide useful convergence criteria for multiple ergodic averages (see for example 
the references mentioned at the end of the introduction).

Let $\mathcal H$ be a Hilbert space and $(u_{n})_{n\in\N^{2}}$ be a
 doubly indexed family of vectors in this space. We will denote, for
any $h\in\Z^{2}$,
$$
\gamma(N,h):=\sum_{\begin{subarray}{1}0<n\leq N\\0<n+h\leq N\end{subarray}}
<u_{n+h},u_{n}>\;,
$$
and
$$
\|u\|_{\infty}:=\sup_{n}\|u_{n}\|\;.
$$
\begin{proposition} \label{gen.vdC.ineq.hilb}
        Let $H\in\N^{2}$ and $(a_{h})_{-H<h<H}$ be a finite
positive-definite family of complex numbers, with $\sum_{h}a_{h}=1$.
     We have
$$
\Big\|\sum_{0<n\leq N}u_{n}\Big\|^{2}\leq N_{1}N_{2} \Big(
\sum_{h}a_{h}\gamma(N,h)
+5\|u\|_{\infty}^{2}\sum_{h}(|h_{1}|N_{2}+|h_{2}|N_{1}
+|h_{1}h_{2}|)|a_{h}|\Big).
$$
\end{proposition}

The proof of Proposition \ref{gen.vdC.ineq.hilb} is similar to the scalar case and will 
be omitted. Combined with Remark \ref{pos-def-crit}, this proposition leads to the following extension of the notion of vdC set to families in Hilbert space.
\begin{corollary}\label{cor.gen.vdC.ineq.hilb}
     Let $D$ be a vdC set in $\Z^{2}$ and $(u_{n})_{n\in\Z^{2}}$ be a
     bounded family in $\mathcal H$. If$$
\forall d\in D,\quad\lim_{N_{1},N_{2}\to+\infty}
\frac1{N_{1}N_{2}}\sum_{0< n\leq (N_{1},N_{2})}<u_{n+d},u_{n}>=0
$$
then
$$
\lim_{N_{1},N_{2}\to+\infty} \frac1{N_{1}N_{2}}\sum_{0<n\leq (N_{1},N_{2})}u_{n}=0\;.
$$
\end{corollary}

\subsection{A new spectral characterization}\label{new-spectral}
We work in this subsection with ordinary sequences indexed by $\Z$. The extension to the
multidimensional case is straightforward. We have the following spectral characterization of vdC sets, which completes the classical Theorem \ref{spectralchar}.

\begin{theorem}\label{n-spectralchar}
Let $D\subset \Z$. Then        $D$ is a van der Corput set if and only if
        any positive measure $\sigma$ on the torus $\T$ such
        that
$
\sum_{d\in D} \left|\widehat\sigma(d)\right|<+\infty
$  is continuous.
\end{theorem}
This result is not surprising. Why?  Because we have a ``parallel" fact pertaining to recurrence properties. It is not difficult to prove that if a set 
$D$ is a set of recurrence, then, for any m.p.s. $(X,\A,\mu,T)$ and any set $A$ in $\A$ 
such that $\mu(A)>0$, not only does there exist $d\in D$ such that 
$\mu(A\cap T^{d}A)>0$, but also $\sum_{d\in D}\mu(A\cap 
T^{d}A)=+\infty$.

\begin{proof}[Proof of Theorem \ref{n-spectralchar}]
Let $D$ be a vdC set in $\Z$, and fix $\epsilon>0$.
By Remark~\ref{pos-def-crit},
we know that there 
exists a positive-definite sequence $(a_{h})_{h\in\Z}$ such that :
\begin{itemize}\item
all but finitely many $a_{h}$ are zero;
\item
$a_{h}=0$ whenever $h\neq0$ and $h\notin D\cup(-D)$;
\item
$a_{0}\leq\epsilon$ and $\sum_{d}a_{d}=1$.
\end{itemize}
Moreover, for any positive-definite sequence $(b_{h})_{h\in\Z}$ 
with support in $\{-H+1,\ldots,H-1\}$ and such that $\sum_{h}b_{h}=1$, 
we have the following vdC 
inequality (simply the one-dimensional version of Theorem \ref{gen.vdC.ineq}): 
for any complex numbers $u_{1},u_{2},\ldots,u_{N}$,
$$
 \left|\sum_{n=1}^Nu_{n}\right|^2\leq 
 N\left(\sum_h b_h\gamma(N,h) +5\|u\|_{\infty}^2\sum_h|hb_h|\right).$$
We apply this inequality to the sequence $(a_h)$ after noticing that 
since the sequence is positive-definite, we 
have $|a_{h}|\leq a_{0}$. We obtain
$$
    \left|\sum_{n=1}^Nu_{n}\right|^2\leq 
 Na_0\left(\gamma(N,0)+\sum_{\begin{subarray}{1}{d\in D\cup(-D)}\\|d|\leq 
H\end{subarray}}\left|\gamma(N,d)\right|+5\|u\|_{\infty}^2H^2\right),
 $$
Hence
\begin{equation}\label{ineq}
 \left|\frac1N\sum_{n=1}^Nu_{n}\right|^2\leq 
 \epsilon\left(\frac1N\sum_{n=1}^N |u_n|^2+\sum_{\begin{subarray}{1}{d\in D\cup(-D)}\\|d|\leq H\end{subarray}}\left|\frac1N\gamma(N,d)\right|+\frac5{N}\|u\|_{\infty}^2H^2\right)
 \end{equation}
 
Let $\sigma$ be a probability measure on the torus such that
$
\sum_{d\in D} \left|\widehat\sigma(d)\right|<+\infty
$.

Following Ruzsa (\cite{R}), we consider a sequence $(Y_{n})_{n\in\N}$ of complex random 
variables of modulus one such that almost surely,
$$
\frac1N\sum_{0< n\leq N}Y_{n}\to\sigma(\{0\})\quad\text{and}\quad
\frac1N\sum_{0< n\leq N}Y_{n+h}\overline{Y_{n}}\to\widehat{\sigma}(h)\;.
$$
(Details of a construction of such a sequence $(Y_n)$ are given below, 
in Lemmas \ref{lemmeprob.1.1} and \ref{lemmeprob.2.1} and in the text which follows 
these lemmas.)\\

We apply (\ref{ineq}) to $u_n=Y_n$ and 
let $N$ go to infinity. After noticing that $$\frac1N\gamma(N,d)=\frac1N\sum_{\begin{subarray}{1}0<n\leq N\\0<n+d\leq N\end{subarray}}Y_{n+d}\overline{Y_{n}}\to\widehat{\sigma}(d)\;,$$ we obtain
$$
|\sigma(\{0\})|^2\leq\epsilon\left(1+2\sum_{d\in 
D}\left|\widehat\sigma(d)\right|\right)\;.
$$
This proves that $\sigma(\{0\})=0$. 
\end{proof}

\section{Enhanced van der Corput sets}\label{enhancedvdcset}
\subsection{Introduction}
In this section, we introduce a new property which we call \emph{enhanced vdC}. It is a
natural concept for several reasons :\begin{itemize}\item
the set of all integers is enhanced vdC, and it is often this property
which is classicaly used in equidistribution theory and ergodic
theory;\item
the spectral characterization of enhanced vdC sets is given by the
FC$^{+}$ property (Theorem \ref{enhanced-sp.char});\item
in the manner that the notion of vdC set is linked to the notion of set of recurrence, the notion of enhanced vdC set is linked to the notion of
set of strong recurrence (see Subsection \ref{enhvdc-sr}). \end{itemize}

We give here the definition and the spectral characterization of
 enhanced vdC sets in $\Z$, extension to $\Z^d$ being completely routine.

\subsection{Definitions and a spectral characterization}\label{enhanced-def-sec}
\label{spectral-char2}
\begin{definition}\label{enhanced-def}
An infinite set of integers $D$ is enhanced van der
Corput if, for any sequence $(u_{n})_{n\in\Z}$ of complex numbers
of modulus 1 such that
\begin{equation}\label{enh-w-lim}
\forall d\in D,\quad\gamma(d):=\lim_{N\to+\infty}
\frac1{N}\sum_{n=0}^{N-1}u_{n+d}\overline{u_{n}}\qquad\text{exists}
\end{equation}and$$
\lim_{|d|\to+\infty, d\in D}\gamma(d)=0\;,
$$
we have
$$
\lim_{N\to+\infty} \frac1{N}\sum_{n=0}^{N-1}u_{n}=0\;.
$$
\end{definition}
(Note that we obtain an equivalent definition if we replace lim by limsup 
in (\ref{enh-w-lim}). See Proposition \ref{def2dvdc}.)
\begin{definition}\label{fcplus-def}
An infinite set of integers $D$ is FC$^{+}$ if every positive
measure $\sigma$ on the torus $\T$ having the property that
$\displaystyle
\ \lim_{|d|\to+\infty, d\in D}\widehat\sigma(d)=0\
$ is continuous.
\end{definition}
This definition appears in \cite{K-MF} and in \cite{JB}. We
remark that in \cite{Peres}, Peres uses the notation FC$^+$ for sets satisfying the apparently weaker Condition (S3) of Theorem \ref{spectralchar}. We ask in Question \ref{first} whether Condition (S3) is actually strictly weaker than Condition FC$^+$.

\begin{theorem}\label{enhanced-sp.char}
The notions of enhanced vdC set and FC$^{+}$ set coincide.
\end{theorem}
\begin{proof} 
      The proof of this proposition follows the lines of the spectral
      characterization of vdC sets. In order to prove that FC$^{+}$ sets
      are  enhanced vdC, we use the following lemma, which is the one parameter version of Lemma~\ref{class_spect}.

      \begin{lemma}\label{affinite}
Let $(u_{n})_{n\in\N}$ be a bounded sequence of complex numbers
and $(N_j)_{j\in\N}$ be an increasing sequence of positive integers.
If for all $h\in\N$
$$
\gamma(h):=\lim_{j\to+\infty}\frac1{N_j}
\sum_{n=1}^{N_{j}}u_{n+h}\overline{u_{n}}\qquad\mbox{exists}\;,
$$
then there exists a positive measure $\sigma$ on the torus
such that, for all $h\in\N$,
$$
\widehat\sigma(h)=\gamma(h)
$$
and this measure satisfies
$$
\limsup_{j\to+\infty}\frac1{N_j}\left|\sum_{n=1}^{N_{j}}u_{n}\right|
\leq\sqrt{\sigma\left(\{0\}\right)}\;.
$$
\end{lemma}

Let $D$ be an FC$^{+}$ set. Let $(u_{n})$
be a bounded sequence of complex numbers such that
$$
\lim_{|d|\to+\infty, d\in D}\lim_{N\to+\infty}
\frac1{N}\sum_{n=1}^{N}u_{n+d}\overline{u_{n}}=0\;.
$$
There exists an increasing sequence $(N_j)_{j\in\N}$ of positive
integers such that
\begin{itemize}
         \item
         $\displaystyle\lim_{j\to+\infty}\frac1{N_{j}}\left|\sum_{n=1}^{N_{j}}
         u_{n}\right|=
\limsup_{N\to+\infty} \frac1{N}\left|\sum_{n=1}^{N}u_{n}\right|$\;,
\item
$\displaystyle\forall h\in \N,\quad \gamma(h):=
\lim_{j\to+\infty}\frac1{N_{j}}\sum_{n=1}^{N_{j}}
u_{n+h}\overline{u_{n}}\qquad\mbox{exists}$\;.
\end{itemize}
The map $\gamma$ is the Fourier transform of a positive measure
$\sigma$ on the torus. We have
$\lim_{|d|\to+\infty, d\in D}\widehat\sigma(d)=0$. By hypothesis, this
forces the measure $\sigma$ to be continuous. We have
$\sigma\left(\{0\}\right)=0$ and, using the above lemma, we obtain
the Ces\`aro convergence of $(u_{n})$ to zero. The set $D$ is  enhanced vdC.

In order to prove that any  enhanced vdC set is FC$^{+}$, the arguments of
Ruzsa (\cite{R}) can be adapted
and we use the following probabilistic lemmas.
\begin{lemma}\label{lemmeprob.1.1}
         Let $(\theta_{n})_{n\in\N}$ be an i.i.d. sequence of random
         variables with values in the torus $\T$. We define a new sequence
         of complex random variables
         $(Y_{n})$ by
         $$
         Y_{n}:=
         e\left(r\theta_m\right)\;,
         $$
         if $n=m^{2}+r$, with $0\leq r\leq2m$.

         We have, almost surely,
         $$
         \lim_{N\to+\infty}\frac1{N}\sum_{n=1}^{N}Y_{n} =
         \P\left(\theta=0\right)\;.
         $$
         \end{lemma}

\begin{lemma}\label{lemmeprob.2.1}
Let $(X_{n})_{n\in\N}$ be an i.i.d. sequence of bounded complex
random variables. We define a new sequence of complex random variables
$(Z_n)$ by
$$
         Z_n:=X_{m}
         $$
         if $n=m^{2}+r$, with $0\leq r\leq2m$.

         We have, almost surely,
          $$
         \lim_{N\to+\infty}\frac1{N}\sum_{n=1}^{N}Z_{n} =\E\left[X\right]\;.
         $$
         \end{lemma}

Let $D$ be an enhanced vdC set, and let $\sigma$ be a positive measure on
$\T$. We suppose that the Fourier coefficient $\widehat\sigma(d)$
goes to zero when $d$ goes to infinity in $D$.
Without loss of generality, we can suppose that $\sigma$ is a
probability measure, and we consider a sequence of independent random
variables
$(\theta_{n})$ of law $\sigma$. We
define, as in Lemma~\ref{lemmeprob.1.1}, the family of complex random
variables $(Y_{n})$. Let us fix $h\in\N$. We define $Z_n=e(h\theta_m)$ for $n=m^2+r $ and $0\leq r\leq2m$. By Lemma \ref{lemmeprob.2.1} we know that, almost surely, $
         \lim_{N\to+\infty}\frac1{N}\sum_{n=1}^{N}Z_{n} =\E\left[e(h\theta)\right]
         $. Furthermore, the set of positive integers $n$ such that $Y_{n+h}\overline{Y_n}=Z_n$ has full density. Thus, almost surely,
$$
\lim_{N\to+\infty}\frac1{N}\sum_{n=1}^{N}
Y_{n+h}\overline{Y_{n}}=\E\left[e(h\theta)\right]\;.
$$

This last quantity is exactly $\widehat{\sigma}(h)$ and, by
hypothesis, it goes to zero when $h$ goes to infinity in $D$.
Since the set $D$ is  enhanced vdC, we conclude that
$$
\lim_{N\to+\infty}\frac1{N}\sum_{n=1}^{N}
Y_{n}=0\;.
$$
By Lemma \ref{lemmeprob.1.1}, this means that
$\P\left(\theta=0\right)=0$, that is to say
$\sigma\left(\{0\}\right)=0$. The same argument can be applied to all
the images of $\sigma$ by translations of the torus, and we conclude
that $\sigma$ is a continuous measure. Hence $D$ is FC$^{+}$.
\end{proof}
The spectral characterization makes it possible to give an alternative
definition of  enhanced vdC sets.
\begin{proposition}\label{def2dvdc}
      An infinite set of integers $D$ is enhanced vdC if and only if for any sequence $(u_{n})_{n\in\Z}$ of complex numbers
of modulus 1 such that
$$
\lim_{|d|\to+\infty, d\in D}\,\limsup_{N\to+\infty}
\left|\frac1{N}\sum_{n=0}^{N-1}u_{n+d}\overline{u_{n}}\right|=0\;,
$$
one has
$$
\lim_{N\to+\infty} \frac1{N}\sum_{n=0}^{N-1}u_{n}=0\;.
$$
\end{proposition}

\subsection{Some properties of enhanced vdC sets}
From the spectral characterization we deduce various corollaries. We omit detailed proofs since they are similar to proofs of the corresponding statements for vdC sets (see Subsection \ref{ex.}).
\begin{corollary}[Ramsey property]
        If $D=D_{1}\cup D_{2}$ is an enhanced vdC set, then at least one
        of the sets $D_{1}$ or $D_{2}$ is  enhanced vdC. 
        \end{corollary}

\begin{corollary}[Sets of differences]
        Let $D\subset\N$. Suppose that, for all $n>0$ there exist $a_1<a_2<\ldots<a_n$ such that $\{a_j-a_i\,:\,1\leq i<j\leq n\}\subset D$. Then $D$ is an enhanced vdC set.
        \end{corollary}
       
       \begin{corollary}[Linear transformations]
Let $d$ and $e$ be positive integers, and $L$ be a linear
transformation from $\Z^d$ into $\Z^e$ (i.e. an $e\times d$ matrix with
integers entries).
\begin{enumerate}
        \item
        If $D$ is an enhanced vdC set in $\Z^d$ and if $0\notin L(D)$, then $L(D)$ is an enhanced vdC set in $\Z^e$.
        \item
        Let $D\subset\Z^d$. If the linear map $L$ is one to one, and if $L(D)$ is an enhanced vdC set
        in $\Z^e$, then $D$ is an enhanced vdC set in $\Z^d$.
\end{enumerate}
\end{corollary}

\begin{corollary}[Lattices are (enhanced vdC)*]
        If $G$ is any $d$-dimensional lattice in $\Z^d$, and if $D$ is an enhanced vdC
        set in $\Z^d$, then $G\cap D$ is an enhanced vdC set in $\Z^d$.
        \end{corollary}

\subsection{Questions}\label{enhancedquestions}
 
\begin{question}\label{first} Our intuition is that there exist vdC sets which are not enhanced vdC.
Is it true? Is it possible to exhibit a particular example?
\end{question}

\begin{question} \label{q2}We know (Corollary \ref{cor.gen.vdC.ineq.hilb}) that
the notions of \emph{vdC set for families in a Hilbert space}
and of vdC set coincide. Is the analogous fact true for enhanced
vdC sets?
\end{question}

\begin{question} \label{q3}We know (Corollary \ref{last}) that any vdC set can be partitionned into infinitely many vdC sets. Is the analogous fact true for enhanced vdC sets?
\end{question}

\begin{question} \label{q4}We know (Corollary \ref{cartesian}) that the Cartesian product of two vdC sets is a vdC set. Is the analogous fact true for enhanced vdC sets?
\end{question}

\subsection{Examples}\label{ex-enh-vdc}
\subsubsection{Ergodic sequences}\label{ergseq}
A sequence of integers $(d_n)_{n\in\N}$ is called ergodic if the following mean ergodic theorem is valid : given an ergodic m.p.s. $(X,\A,\mu,T)$ and $f\in L^2(\mu)$, the averages $\frac1N\sum_{n=1}^N f\circ T^{d_n}$ converge in $L^2$ to $\int f\,\text{d}\mu$ when $N$ goes to infinity.

It follows from the spectral theorem that the sequence $(d_n)$ is ergodic if and only if, for all $x\in\R\setminus\Z$,
\begin{equation}\label{ergodic}
\lim_{N\to+\infty}\frac1N\sum_{n=1}^N e(d_nx)=0\;.
\end{equation}

\begin{proposition}\label{erg-seq}
Any ergodic sequence is an enhanced vdC sequence.
\end{proposition}
\begin{proof}
Let $(d_n)$ be an ergodic sequence and $\sigma$ a finite measure on the torus. Using the dominated convergence theorem we deduce from (\ref{ergodic}) that
$$
\lim_{N\to+\infty}\frac1N\sum_{n=1}^N\widehat\sigma(d_n) =\sigma\left(\{0\}\right)\;.
$$
Hence it is immediate that the sequence $(d_n)$ is FC$^+$.
\end{proof}

Propostion \ref{erg-seq} can be used to exhibit many examples of enhanced vdC sets. 

\medbreak
(i) In \cite{B-K-Q-W} the authors consider sequences of the form $d_n=[a(n)]$ where the 
function $a$ belongs to some Hardy field. They characterize those of them which are 
ergodic. See Theorems 3.2, 3.3, 3.4, 3.5 and 3.8 in \cite{B-K-Q-W}.
Here are some examples of ergodic sequences, coming from \cite{B-K-Q-W}.
\begin{itemize}\item
$\{[bn^c]\,:\,n\in\N\}$, where $c$ is irrational $>1$ and $b\neq0$.
\item
$\{[bn^c+dn^a]\,:\,n\in\N\}$, where $b,d\neq0$, $b/d$ is irrational, $c\geq1$, $a>0$ and $a\neq c$.
\item
$\{[bn^c(\log n)^d]\,:\,n\in\N\}$, where $b\neq0$, $c$ is irrational $>1$ and $d$ is any number.
\item
$\{[bn^c(\log n)^d]\,:\,n\in\N\}$, where $b\neq0$, $c$ is rational $>1$ and $d\neq0$.
\item
$\{[bn^c+d(\log n)^a]\,:\,n\in\N\}$, where $b,d\neq0$, $c\geq1$, and $a>1$.
\end{itemize}

The paper \cite{B-K-Q-W} contains also interesting examples of non ergodic sequences. For example the sequence $[\sqrt2 n^{3/2} + \log n]$ is not ergodic, whereas 
sequences $[\sqrt2 n^{3/2} + (\log n)^2]$ and $[\sqrt2 n^{\pi/2} + \log n]$ are ergodic. Is $\{[\sqrt2 n^{3/2} + \log n]\,:\,n\in\N\}$ an enhanced vdC set ? We leave this as an open question.

\medbreak
(ii) In \cite{VB-IH2} a mean ergodic theorem along a \emph{tempered sequence} is proved. More precisely, it is shown (see Theorem 8.1 in \cite{VB-IH2}) that, for any \emph{tempered function}\footnote{A real valued function $g$ defined on a half line $[\alpha,+\infty)$ is called a tempered function if there exist $k\in\N$ such that $g$ is $k$ times continuously differentiable, $g^{(k)}(x)$ tends monotonically to zero as $x\to+\infty$, and $\lim_{x\to+\infty}x\left|g^{(k)}(x)\right|=+\infty$. This notion is classical in the theory of uniform distribution, see \cite{cigler}.} $g$, the sequence $([g(n)])$ is ergodic. 
This gives a new large class of examples. For example, the function $g(x)=x^a\left(\cos((\log x)^b)+2\right)$, where $a>0$ and $0<b<1$ is a tempered function (which does not belong to any Hardy field).

\medbreak
(iii) A different type of example is provided by so-called ``automatic sequences''. 
Characterizations of ergodic automatic sequences are well known (see for example 
\cite{mauduit}). A typical example of such a sequence is the Morse sequence 
$(0,3,5,6,9,10,\ldots)$, which is the sequence of integers the sum of whose digits in 
base two is even.

\medbreak
(iv) As a consequence of the Wiener-Wintner ergodic theorem, we know that for any 
weakly mixing m.p.s. $(X,\A,\mu,T)$ and for any $A\in\A$ with $\mu(A)>0$, for almost 
every $x\in X$ the sequence $\left\{n\in\N\,:\,T^nx\in A\right\}$  is ergodic.

In \cite{LLPVW} other types of random sequences that are almost surely ergodic are 
constructed, namely sequences of the form $\left(\sum_{n=0}^{N-1} f\circ T^n\right)$
where $f$ is an integer valued function on a m.p.s. $(X,\A,\mu,T)$, under conditions on the m.p.s. and the function.

\subsubsection{Polynomial sequences}\label{polex}

The examples given in Subsection \ref{pol+prime} not only have the ordinary vdC-property, 
but they have also the enhanced vdC property (in $\Z^d$). We restrict ourselves here to 
the one-parameter case.

The following criterion, which generalizes Proposition \ref{erg-seq}, 
is useful in obtaining additional interesting examples.
\begin{proposition}\label{source}
Let $D=(d_n)_{n\in\N}$ be a sequence of nonzero integers.
Suppose that 
\begin{enumerate}\item[(i)] For all $q\in\N$, $D\cap q\Z$ has positive upper density in $D$;
\item[(ii)] For all irrational real numbers $x$, the sequence $(d_nx)$ is 
uniformly distributed mod 1.\end{enumerate}
Then $D$ is an enhanced vdC sequence.
\end{proposition}

\begin{proof}
Fix $q\in\N$. There exists an increasing sequence of positive integers $\left(N_k^{(q)}\right)_{k\in\N}$ such that
\begin{equation}\label{density}
\liminf_{k\to+\infty}\frac1{N_k^{(q)}}\#\left\{n\in \left[1,N_k^{(q)}\right]\,:\,q!\text{ divides }d_n\right\}\ >0\;.\end{equation}

Let us define a family of uniformly bounded trigonometric polynomials with spectrum
contained in $D$, by the formula
\begin{equation}\label{trig-p}
P_{q,k}(x):=\frac1{\#\{n\leq N_k^{(q)},\,q!|d_n\}}\sum_{n\leq N_k^{(q)},\,q!|d_n}e\left(d_n x\right)\;.
\end{equation}

Replacing if necessary the sequence $\left(N_k^{(q)}\right)$ by a subsequence, we can 
suppose that, for all rational numbers $y$, the sequence $\left(P_{q,k}(y)\right)$ 
converges as $k\rightarrow+\infty$.

Consider now an irrational real number $x$. We have
\begin{align*}
P_{q,k}(x)&=&\frac1{\#\{n\leq N_k^{(q)},\,q!|d_n\}}\sum_{n\leq N_k^{(q)}}e\left(d_n x\right)\frac1{q!}\sum_{j=0}^{q!-1}e(d_nj/q!)\\&=&
\frac{N_k^{(q)}}{\#\{n\leq N_k^{(q)},\,q!|d_n\}}\frac1{q!}\sum_{j=0}^{q!-1}
\frac1{N_k^{(q)}}\sum_{n\leq N_k^{(q)}}e\left(d_n\left(x+\frac{j}{q!}\right)\right).
\end{align*}
Using (\ref{density}) and hypothesis (ii), we see that $\lim_{k\to+\infty}P_{q,k}(x)=0$.

We denote by $g_q$ the pointwise limit of the sequence $(P_{k,q})_{k\in\N}$. For all 
rational numbers $y$, we have $P_{q,k}(y)=1$ for all large enough $q$. 

Letting $q$ go to infinity, we see that the sequence $(g_q)$ converges everywhere to 
the characteristic function of the rationals. Applying the dominated convergence theorem 
twice, we observe that, for all finite measures $\sigma$ on $\T$,
$$
\lim_{q\to+\infty}\lim_{k\to+\infty}\int_{\T}P_{q,k}\,\text{d}\sigma =\sigma \left(\Q/\Z\right).
$$

Let $\sigma$ be a positive measure on $\T$ such that $\lim_{n\to+\infty}\widehat
\sigma(d_n)=0$. From (\ref{trig-p}), we deduce that $\lim_{k\to+\infty}\int_{\T}P_{q,k}\,
\text{d}\sigma =0$, hence $\sigma \left(\Q/\Z\right)=0$, and in particular $\sigma(\{0\})=0$. 

We have proved that $D$ is an FC$^+$ set.

\end{proof}

From Proposition \ref{source}, one can deduce the following (not too surprising) corollaries.
\begin{corollary}
Let $p$ be a polynomial with integer coefficients.
The sequence $(p(n))_{n\in\N}$
is enhanced vdC if and only if for all positive integers $q$, there
exists $n\geq1$ such that $q$ divides $p(n)$.
\end{corollary}

\begin{corollary}
Let $f$ be a (non zero) polynomial with integer coefficients and zero constant term. Sequences $\{(f(p-1)): p\in P\}$ and $\{(f(p+1)):p\in P\}$ are enhanced vdC.
\end{corollary}

Let us describe one more family of examples, coming from \emph{generalized polynomials}\footnote{The class of polynomial functions is obtained, starting from the constants and the identity function $x\mapsto x$, by the use of addition and multiplication. To define the class of generalized polynomials just add the greatest integer function as an allowed operation.}, dealt with in \cite{VB-IH}. Let $q$ be an integer valued generalized polynomial. Corollary 3.5 of \cite{VB-IH} gives a sufficient condition for the sequence $(q(n))$ to be an \emph{averaging sequence of recurrence} and this condition is the same as the hypothesis of our Proposition \ref{source}. In particular averaging sequences of recurrence in \cite{VB-IH} (see page 106), provide examples of enhanced vdC sets. Here are two of these examples.
\begin{itemize}\item
For $\alpha_1,\alpha_2,\ldots,\alpha_k$ non zero real numbers and $k\geq3$,$$
\{[\alpha_1n][\alpha_2n]\ldots[\alpha_kn]\,:\,n\in\N\} \quad\text{is an enhanced vdC-set.}
$$\item
For $\alpha$ a non zero real number,
$
\{[\alpha n]n^2\,:\,n\in\N\} $ is an enhanced vdC-set.
\end{itemize}

\section{Van der Corput sets and sets of recurrence}\label{vdc-recur}
In this section we discuss some links between the vdC property and recurrence in dynamical systems. 
\subsection{Sets of strong recurrence}\label{def-rec-sec}

        Recall that
a subset $D$ of $\Z$ is a \emph{set of recurrence} if, given any m.p.s.
$(X,\A,\mu,T)$ and
any subset $A$ in $\A$ of positive $\mu$-measure, there exists $d\in
D,\,d\neq0$ such that $\mu\left(A\cap T^dA\right)>0$.
\begin{definition}\label{rec-def}
An infinite subset $D$ of $\Z$ is a \emph{set of  strong recurrence}
 if, given any m.p.s.
$(X,\A,\mu,T)$ and
any subset $A$ in $\A$ of positive $\mu$-measure,
$$\limsup_{d\in D, |d|\to+\infty}\mu\left(A\cap T^dA\right)>0\;.$$
\end{definition}

One of the reasons to be interested in sets of strong recurrence is that they naturally appear in combinatorial applications. See for example Theorem~4.1 in \cite{bergelson}.

Alan Forrest (\cite{AF}) gave an example of a set of recurrence
which is not a set of strong recurrence.

\subsection{VdC sets and sets of recurrence}
Recall once more the definition of a vdC set (cf. Definition \ref{vdc-def}).

A set of non zero integers $D$ is a van der
Corput set if, for any sequence $(u_{n})_{n\in\N}$ of complex numbers
of modulus 1 such that
$$
\forall d\in D,\quad\gamma(d):=\lim_{N\to+\infty}
\frac1{N}\sum_{n=1}^{N}u_{n+d}\overline{u_{n}}=0\;,
$$
we have
$$
\lim_{N\to+\infty} \frac1{N}\sum_{n=1}^{N}u_{n}=0\;.
$$

We know that we obtain an equivalent definition if we replace in the
last sentence ``any sequence $(u_{n})_{n\in\N}$ of complex numbers
of modulus 1'' by ``any bounded sequence $(u_{n})_{n\in\N}$ of complex
numbers''. (This is a consequence of the generalized vdC inequality,
as Corollary \ref{cor.gen.vdC.ineq.hilb} follows from Proposition~\ref{gen.vdC.ineq.hilb}.)

A set $D$ is a set of recurrence if and only if it is
\emph{intersective}, namely satisfies the following condition: for any set of integers $E$ of
positive upper density, one has $D\cap(E-E)\neq\emptyset$. This fact is well known (see
\cite{Anne} and \cite{bergelson}). This fact is utilized in the proof of the following theorem.

\begin{theorem}\label{vdc01}Let $D\subset\Z\setminus\{0\}$.
The set $D$ is a set of recurrence if and only if it satisfies the following van
der Corput's type property :
for any sequence $(u_{n})_{n\in\N}$ of $0$'s and $1$'s such that
$$
\forall d\in D,\quad\gamma(d):=\lim_{N\to+\infty}
\frac1{N}\sum_{n=1}^{N}u_{n+d}u_{n}=0
$$
we have
$$
\lim_{N\to+\infty} \frac1{N}\sum_{n=1}^{N}u_{n}=0\;.
$$
\end{theorem}

It is an exercise to verify that we obtain an equivalent statement if we
replace in the preceding sentence
``for any sequence $(u_{n})_{n\in\N}$ of $0$'s and $1$'s'' by
``for any bounded sequence $(u_{n})_{n\in\N}$ of positive real numbers''.

As a consequence of Theorem \ref{vdc01}, we obtain the well known
fact that any van der Corput set is a set of recurrence (\cite{K-MF}). Answering a
question of Ruzsa, Bourgain proved in \cite{JB} that there exist
sets of recurrence which are not vdC.

\begin{proof}[Proof of Theorem \ref{vdc01}]
If $D$ is not a set of recurrence, then there exists a set $E\subset\N$ such that
$$
\overline{d}(E):=\limsup_{N\to+\infty}\frac1N|E\cap[1,N]|>0
\quad\text{
and
}\quad
D\cap(E-E)=\emptyset\;.
$$
If we consider the sequence $(u_{n})$ defined by
$$
u_{n}=1\ \text{if}\ n\in E\ ,\quad u_{n}=0\ \text{if}\ n\notin E\;,
$$
we see that
$$
\forall d\in D,\quad
\frac1{N}\sum_{n=1}^{N}u_{n+d}u_{n}=0\;,
\quad\text{
but
}\quad
\limsup\frac1{N}\sum_{n=1}^{N}u_{n}>0\;.
$$
This proves the ``if'' part of the Theorem.

Suppose now that $D$ is a set of recurrence.
The fact that if $E$ is a set of
positive upper density, then there exists $d\in D$ such that
$\{n\in E\,:\,n+d\in E\}\neq\emptyset$ is a consequence of
Furstenberg's correspondence
principle. But this principle gives more\footnote{For a statement of Furstenberg's correspondence
principle in the form we utilize here, see for example Theorem 1.1 in \cite{VB3}.} : there exists $d\in D$ such that
the set $\{n\in E\,:\,n+d\in E\}$ has positive upper density.

Hence if a sequence $(u_{n})$ is the indicator of a set $E$ of
positive upper density, then there exists $d\in D$ such that
$$
\limsup_{N\to+\infty}
\frac1{N}\sum_{n=1}^{N}u_{n+d}u_{n}>0\;.
$$

\end{proof}

The similarity and the distinction between the recurrence property and the
vdC property is also illustrated by the next proposition (to be
compared with the spectral characterization of vdC sets - Theorem \ref{spectralchar}). 

If $(X,\A,\mu,T)$ is a m.p.s. and if $A\in\A$, we denote by $\sigma_A$ the \emph{spectral measure of $A$}, which is defined by $\mu\left(A\cap T^{-n}A\right)=\widehat\sigma_A(n)$, for any $n\in\Z$. If $f$ is a square integrable function on $X$, we denote by $\sigma_f$ the \emph{spectral measure of $f$}, which is defined by $\int f\circ T^n\cdot f\,\text{d}\mu=\widehat\sigma_f(n)$ , for any $n\in\Z$. (Of course, we have $\sigma_A=\sigma_{{\bf1}_A}$.)

\begin{proposition}\label{spectral01} Let $D\subset\Z\setminus\{0\}$.
The set $D$ is a set of recurrence if and only if one of the two equivalent following properties is satisfied:\begin{itemize}\item
In any ergodic m.p.s., if the Fourier transform $\widehat{\sigma_A}$ of a set $A$ vanishes on $D$, then $\sigma_A=0$.
\item
In any ergodic m.p.s., if the Fourier transform $\widehat{\sigma_f}$ of a bounded positive function $f$ vanishes on $D$, then $\sigma_f=0$.
\end{itemize}
\end{proposition}

 \begin{proof}
Suppose that $D$ is not a set of recurrence. There exists an ergodic m.p.s. $(X,\A,\mu,T)$ and a set $A$ in $\A$, with positive measure such that,
for all
$d\in D$, $\mu\left(A\cap T^{d}A\right)=0$. The spectral measure $\sigma_A$ of the set $A$ satisfies $\widehat\sigma(d)=0$ for all $d\in D$, and $\sigma_A(\{0\})=\mu(A)\neq0.$

Suppose that $D$ is a set of recurrence. Let $\sigma_f$ be the spectral measure of a bounded positive function $f$. Suppose that for all $d\in D$, $\widehat{\sigma_f}(d)=0$. By the ergodic theorem, we have almost surely, for all $d\in D$,
$$
0=\int f\cdot f\circ T^d\,\text{d}\mu=\lim_{N\to+\infty}\frac1N\sum_{n=0}^{N-1}f\circ T^n\cdot f\circ T^{n-d}.
$$
Using Theorem \ref{vdc01} (and more precisely the remark immediately following the theorem), we obtain that, almost surely,
$$
\lim_{N\to+\infty}\frac1N\sum_{n=0}^{N-1}f\circ T^n =0\;.
$$
The ergodic theorem gives $\int f\,\text{d}\mu=0$, hence $\sigma_f=0$.
\end{proof}
\subsection{Enhanced vdC sets and strong recurrence}\label{enhvdc-sr}
The results in this subsection indicate that the link between  enhanced van der
Corput sets and sets of strong recurrence
is parallel to the link between van der Corput sets and sets of
recurrence. However, we don't know if there exists here any example of
Bourgain's type (\cite{JB}). Such an example would give a negative answer to the following question.

\begin{question} (\emph{perhaps very difficult})
Is every set of strong recurrence an FC$^+$ set (or, equivalently, an enhanced
van der Corput set)?
\end{question}

The following question also comes naturally.
\begin{question}
Is there any inclusion between the collection of sets of strong recurrence  and the collection of van der Corput sets ?
\end{question}

The next theorem gives an equivalence between strong recurrence
and \emph{strong intersectivity} (which is defined by (SR2) below).

\begin{theorem}\label{dbl-inters}
Let $D\subset\Z$. There is equivalence between the following assertions.
\begin{enumerate}\item[(SR1)]
      $D$ is a set of strong recurrence.
      \item[(SR2)]
      For any $E\subset\N$ of upper density $\overline{d}(E)>0$,
      there exists
      $\epsilon>0$ and infinitely many $d\in D$ such that
      $$
      \overline{d}\left(E\cap (E+d)\right)>\epsilon\;.
      $$
      \item[(SR3)]
      For any sequence $(u_{n})_{n\in\N}$ of $0$'s and $1$'s such that
$$
\lim_{|d|\to+\infty,d\in D}\,\limsup_{N\to+\infty}
\frac1{N}\sum_{n=1}^{N}u_{n+d}\overline{u_{n}}=0
$$
we have
$$
\lim_{N\to+\infty} \frac1{N}\sum_{n=1}^{N}u_{n}=0\;.
$$
\end{enumerate}
\end{theorem}

\begin{proof}[Proof of Theorem \ref{dbl-inters}]

It is clear that properties (SR2) and (SR3) are the same.
The fact that (SR1)$\Rightarrow $(SR2) follows directly from Furstenberg's
correspondence principle.
The following proof of (SR2)$\Rightarrow $(SR1) has been communicated to
us by Anthony Quas .
Let $(X,\A,\mu,T)$ be a m.p.s. and $A\in\A$, with
$\mu(A)>0$. Let $(x_{n})_{n\geq1}$ be a sequence of random points in $X$ chosen
independently and
with the law $\mu$. We consider a new sequence in $X$ defined by
$$
(y_{n}):=(x_{1},x_{2},Tx_{2},x_{3},Tx_{3},T^{2}x_{3},
x_{4},\ldots,T^3x_{4},x_{5},\ldots,T^{4}x_{5},\ldots)\;,
$$
and the random set $E$ of numbers $n$ such that $y_{n}\in A$.
We claim that, almost surely,
\begin{equation}\label{E-dens}
\lim_{N\to+\infty}\frac1N\sum_{n=1}^{N}{\bf1}_{E}(n)=\mu(A)\;.
\end{equation}
This claim can be justified by the following law of large numbers,
applied to the mutually independent random variables
$$
Y_{k}:=\left(\frac1k\sum_{j=0}^{k-1}{\bf1}_{A}(T^jx_{k})\right)
-\mu(A)\;.$$
\begin{lemma}[law of large numbers]
      Let $(Y_{k})$ be a sequence of random variables
      such that $\sup_{k}\E(Y_{k}^{2})<+\infty$, $\E(Y_{k})=0$, and
      $\E(Y_{k}Y_{\ell})=0$ if $k\neq\ell$. Almost surely we have
      \begin{equation}\label{orth-rv}
      \lim_{n\to+\infty}\frac1{n^{2}}\sum_{k=1}^n kY_{k}=0\;.
      \end{equation}
\end{lemma}
(The convergence (\ref{orth-rv}) is a direct consequence of some easy $L^2$ estimates. It can be also deduced from the convergence of ordinary Ces\`aro averages. We omit the proof.)

Using the block structure of the sequence $(y_{n})$ a
similar argument gives (almost surely)
\begin{equation}\label{EetE+d-dens}
\lim_{N\to+\infty}\frac1N\sum_{n=1}^{N}{\bf1}_{E}(n)
{\bf1}_{E}(n+d)=\mu(A\cap T^{-d}A)\;.
\end{equation}

Assume now that condition (SR2) is satisfied. From (\ref{E-dens}) we
deduce that $\overline{d}(E)>0$, hence there exists
      $\epsilon>0$ and infinitely many $d\in D$ such that
      $$
      \overline{d}\left(E\cap (E+d)\right)>\epsilon\;,
      $$
which means (by (\ref{EetE+d-dens})) that $\mu(A\cap T^{-d}A)>\epsilon$.
\end{proof}

\begin{proposition}\label{e-vdc-a-rec}
      Any  enhanced vdC set is a set of strong recurrence.
      \end{proposition}

      \begin{proof}
	Let $D\subset\Z$ be an enhanced vdC set, let $(X,\A,\mu,T)$ be a m.p.s. and $A\in\A$, with
	$\mu(A)>0$. There exists a positive measure $\sigma$ on the torus
	such that, for all $n\in\Z$,
	$$
	\widehat{\sigma}(n)=\mu(A\cap T^{n}A)\;.
	$$
	This measure has a point mass at zero : $\sigma(\{0\})\geq\mu(A)^{2}$.
	Since the
	set $D$ is FC$^{+}$, this implies that there exists $\epsilon>0$
	such that $\widehat{\sigma}(d)>\epsilon$ for infinitely many $d\in D$.
\end{proof}

\subsection{Density notions of vdC sets and sets of recurrence}

A new natural notion of vdC-type set, which we will call \emph{density vdC} can be 
obtained by replacing in Definition \ref{enhanced-def} the convergence of $\gamma$ to 
zero along the set $D$ by the convergence of $\gamma$ to zero along a subset of $D$ 
which has full density in $D$. We will associate to it a notion of 
\emph{density FC$^+$ set}. These notions are related to \emph{averaging sets of recurrence}, 
as we will see below. Here are the formal definitions.

If $D$ is an infinite set of integers, we will write $D=\{d_m\,:\,m\in\N\}$ with the convention that the numbers $d_m$ are pairwise distinct and the sequence $(|d_m|)$ is nondecreasing. Let us recall that for any bounded sequence $ \left(v(d_m)\right)_{m\in\N}$ of positive numbers the two following properties are equivalent:\begin{itemize}\item$\displaystyle \lim_{M\to+\infty}\frac1M\sum_{m=1}^Mv(d_m)=0$,\item There exists $D'\subset D$ such that
$$
\lim_{M\to+\infty}\frac{\#D'\cap[-M,M]}{\#D\cap[-M,M]}=1\qquad\text{and}\qquad \lim_{m\to+\infty, d_m\in D'}v(d_m)=0.
$$
\end{itemize}
\begin{definition}
An infinite set of integers $D$ is a density vdC set if for any sequence $(u_{n})_{n\in\Z}$ of complex numbers
of modulus 1 such that
$$
\lim_{M\to+\infty}\frac1M\sum_{m=1}^M\limsup_{N\to+\infty}
\left|\frac1{N}\sum_{n=0}^{N-1}u_{n+d_m}\overline{u_{n}}\right|=0\;,
$$
one has
$$
\lim_{N\to+\infty} \frac1{N}\sum_{n=0}^{N-1}u_{n}=0\;.
$$
\end{definition}

(Compare this definition with Proposition \ref{def2dvdc}.)

\begin{definition}
An infinite set of integers $D$ is a density FC$^{+}$ set if every positive
measure $\sigma$ on the torus $\T$ such that
$\displaystyle
\ \lim_{M\to+\infty}\frac1M\sum_{m=1}^M\widehat\sigma(d_m)=0\
$ is continuous.
\end{definition}
(Compare with Definition \ref{fcplus-def}. Any density FC$^+$ set is a FC$^+$ set.)

\begin{definition}
      An infinite set of integers $D$ is an \emph{averaging set of
      recurrence} if for any m.p.s. $(X,\A,\mu,T)$ and $A\in\A$, with $\mu(A)>0$,
      $$
      \limsup_{M\to+\infty} \frac1M\sum_{m=1}^M\mu\left(A\cap
      T^{-d_{m}}A\right)>0\;.
      $$
      \end{definition}
Note that this definition differs slightly from the one given in \cite{VB-IH} where the limsup is replaced by a lim. 

Any averaging set of recurrence is a set of srong recurrence.

\begin{theorem}\label{spec-carac-dens}
The notions of a density vdC set and of a density FC$^+$ sets coincide.
\end{theorem}

The proof of this theorem is similar to the proof of Theorem \ref{enhanced-sp.char} and is omitted.\medbreak

From Theorem \ref{spec-carac-dens} one can deduce for example that the class of density vdC sets has 
the Ramsey property.\medbreak

Of course every density vdC set is an enhanced vdC set. We do not know whether the 
reverse implication holds.

\begin{question}
Do the notions of density vdC set and enhanced vdC set coincide ?
\end{question}\medbreak

Questions \ref{q2}, \ref{q3} and \ref{q4} that we asked about enhanced vdC sets 
have obvious density vdC sets analogues.

Note also that the examples described in Subsection \ref{ex-enh-vdc} can also be 
utilized to illustrate the notion of density vdC set. In particular we have:
\begin{itemize}
\item
If $(d_n)$ is an increasing ergodic sequence of integers, then the set $\{d_n\}$ is a 
density vdC set.  This leads to the examples presented in Subsection~\ref{ergseq}.
\item
If an increasing sequence of integers $(d_n)$ satisfies hypotheses (i) and (ii) of 
Proposition \ref{source}, then the set $\{d_n\}$ is a density vdC set.  This leads 
to the ``polynomial examples" presented in Subsection~\ref{polex}.
\end{itemize}

The following proposition establishes a link with recurrence.
\begin{proposition}
Any density vdC set is an averaging set of recurrence.
\end{proposition}

The proof of this proposition is similar to the proof of Proposition \ref{e-vdc-a-rec}, and is omitted.\\

\subsection{Nice vdC sets and nice recurrence}

Another natural notion of recurrence is that of nice recurrence.

\begin{definition}
      A set $D$ of integers is a \emph{set of nice recurrence} if given
      any m.p.s. $(X,\A,\mu,T)$ and $A\in\A$, with $\mu(A)>0$, given any
      $\epsilon>0$, we have
      $$
      \mu\left(A\cap T^{-d}A\right)\geq\mu(A)^{2}-\epsilon\;,
      $$
      for infinitely many $d\in D$.
\end{definition}

The following proposition provides an equivalent definition for sets of nice recurrence.
\begin{proposition}\label{alt-def}
      A set $D$ of integers is a set of nice recurrence if and only if the following is true:
      
      {\rm(C)} given
      any m.p.s. $(X,\A,\mu,T)$ and $A\in\A$, with $\mu(A)>0$, given any
      $\epsilon>0$, there exists $d\in D,\,d\neq0$ such that
      $
      \mu\left(A\cap T^{-d}A\right)\geq\mu(A)^{2}-\epsilon\;.
      $
\end{proposition}
\begin{proof}
We have to prove that the integer $d$ appearing in Condition (C)
can be chosen arbitrarily large. We suppose that Condition (C) is
satisfied. We consider a m.p.s. $(X,\A,\mu,T)$ and a set $A\in\A$, with $\mu(A)>0$. Denote by
$(Y,\B,\nu,S)$ a Bernoulli scheme on two letters ($Y$ is the set of
sequences of 0's and 1's, $\nu$ is a non trivial product measure,
and $S$ is the shift). Let $k$ be a positive integer and $B$ be the
cylinder set in $Y$ of all sequences beginning by a 1 followed by $k$
0's. We have $\nu(B)>0$, $\nu(B\cap S^{-d}B)=0$ if $|d|\leq k$, and
$\nu(B\cap S^{-d}B)=\nu(B)^{2}$ if $|d|> k$. Applying the hypothesis
to the product $T\times S$ of the two dynamical systems, we affirm
that there exists $d\in D$ such that
$$
\mu\otimes\nu\big((A\times B)\cap (T\times S)^{-d}(A\times B)\big)
\geq\big(\mu\otimes\nu(A\times B)\big)^{2}-\epsilon\nu(B)^{2}\;,
$$
hence there exists $d\in D$, $|d|>k$, such that
$$
      \mu\left(A\cap T^{-d}A\right)\geq\mu(A)^{2}-\epsilon\;.
      $$
      \end{proof}

The notion of sets of nice recurrence
seems to be naturally related to the following definitions.

\begin{definition}
      An infinite set $D$ of integers is a \emph{nice vdC set} if,
      for any
      sequence $(u_{n})_{n\in\N}$ of complex numbers of modulus one,
      $$
      \limsup_{N\to+\infty}\left|\frac1N\sum_{n=1}^{N}u_{n}\right|^{2} \leq
      \limsup_{|d|\to+\infty,\,d\in
      D}\limsup_{N\to+\infty}\left|\frac1N\sum_{n=1}^{N}u_{n+d}
      \overline{u_{n}}\right|\;.
      $$
\end{definition}
\begin{definition}
      A infinite set $D$ of integers is a \emph{nice FC$^{+}$ set} if,
      for any positive measure $\sigma$ on the torus,
      $$
      \sigma(\{0\})\leq\limsup_{|d|\to+\infty,\,d\in
      D}|\widehat{\sigma}(d)|\;.
      $$
\end{definition}

The following proposition is similar in spirit to Proposition \ref{alt-def}.

\begin{proposition}      A set $D$ of integers is a nice FC$^{+}$ set if and only if the following is true:
      
      {\rm (C')} for any positive measure $\sigma$ on the torus and any
      $\epsilon>0$, there exists $d\in D,\,d\neq0$ such that
      $
      |\widehat{\sigma}(d)|>\sigma(\{0\})-\epsilon\;.
      $
\end{proposition}
\begin{proof}We have to prove that the integer $d$ appearing in Condition (C')
can be chosen arbitrarily large. We suppose that Condition (C') is
satisfied. Let $k$ be a positive
integer. There exists a positive
measure $\rho$ on the torus such that $\widehat\rho(n)=0$
if $|n|\leq k$ and $\widehat\rho(n)=\rho(\{0\})>0$ if $|n|>k$. (Choose
the spectral measure of the indicator of the set $B$ in the Bernoulli
scheme considered in the proof of Proposition \ref{alt-def}.) We apply
our hypothesis to the measure $\sigma\star\rho$. There exists $d\in
D$ such that
$$
|\widehat{\sigma}(d)\widehat{\rho}(d)|=|\widehat{\sigma\star\rho}(d)|>
     \sigma\star\rho(\{0\})-\epsilon\rho(\{0\})\geq
     \sigma(\{0\})\rho(\{0\})-\epsilon\rho(\{0\})\;,
     $$
     hence there exists $d\in D$, $|d|>k$, such that
$$
      |\widehat{\sigma}(d)|>\sigma(\{0\})-\epsilon\;.
      $$
      \end{proof}
\begin{question} What are the implications between the three properties :
nice vdC, nice FC$^{+}$ and nice
recurrence ?
\end{question}
Here is what we know:
\begin{enumerate}
      \item[(N1)]
      Nice FC$^{+}$ $\Rightarrow$ nice recurrence.
      \item[(N2)]
      Nice FC$^{+}$ $\Rightarrow$ nice vdC.
      \item[(N3)]
      Nice vdC $\Rightarrow$ a weak form of nice recurrence. \\Here is what this last assertion means.      Let $D$ is a nice vdC set;      for any probability measure $\sigma$ on the torus,
      $$
      \sigma(\{0\})^{2}\leq\limsup_{|d|\to\infty,d\in
      D}|\widehat{\sigma}(d)|\;,
      $$
      and, consequently, we have the following recurrence property:\\
     given
      any m.p.s. $(X,\A,\mu,T)$ and $A\in\A$, with $\mu(A)>0$, given any
      $\epsilon>0$, we have
      \begin{equation}\label{strange}
      \mu\left(A\cap T^{-d}A\right)\geq\mu(A)^{4}-\epsilon\;,
      \end{equation}
      for infinitely many $d\in D$. \\
      (Note that the exponent 4 in (\ref{strange}) is not a
typo. It would be ``nice" to better understand the meaning of inequality (\ref{strange}).)
      \end{enumerate}

The proof of (N1) is a direct application of the spectral theorem : let
$(X,\A,\mu,T)$ be a m.p.s. and $A\in\A$. There exists a positive measure $\sigma$ on
      the torus such that
      $$
      \forall n\in\N,\quad \widehat{\sigma}(n) =\mu\left(A\cap
      T^{-n}A\right) \quad\text{and}\quad\sigma(\{0\})=
      \int_{A}\mu\left(A|\I\right)\,\text{d}\mu\geq\mu(A)^{2}\;.
      $$

The proof of (N2) follows the line of the spectral characterization
described in Subsections \ref{spectral-char1} and \ref{spectral-char2}.
Let $(u_{n})$ be a sequence of complex numbers of modulus one and
$$
M:=\limsup_{|d|\to\infty,d\in
      D}\limsup_{N\to+\infty}\left|\frac1N\sum_{n=1}^{N}u_{n+d}
      \overline{u_{n}}\right|\;.
      $$
There exists an increasing sequence $(N_j)_{j\geq 0}$ of positive
integers such that
\begin{itemize}
         \item
         $\displaystyle\lim_{j\to+\infty}\frac1{N_{j}}\left|\sum_{n=1}^{N_{j}}
         u_{n}\right|=
\limsup_{N\to+\infty} \frac1{N}\left|\sum_{n=1}^{N}u_{n}\right|$\;,
\item
$\displaystyle\forall h\in \Z,\quad \gamma(h):=
\lim_{j\to+\infty}\frac1{N_{j}}\sum_{n=1}^{N_{j}}
u_{n+h}\overline{u_{n}}\qquad\mbox{exists}$\;.
\end{itemize}
The map $\gamma$ is the Fourier transform of a positive measure
$\sigma$ on the torus. Suppose that $D$ is a nice vdC set.
By Lemma \ref{affinite} we have
$$
\limsup_{N\to+\infty}
\left|\frac1{N}\sum_{n=1}^{N}u_{n}\right|^{2}\leq \sigma(\{0\}) \leq
\limsup_{|d|\to\infty,d\in
      D}|\widehat{\sigma}(d)|\leq M\;.
      $$

Claim (N3) can be proved using Lemmas \ref{lemmeprob.1.1} and
\ref{lemmeprob.2.1}. Following the method described in
Subsection \ref{spectral-char2}, we have
$$
\lim_{N\to+\infty}\frac1{N}\sum_{n=1}^{N}
Y_{n+h}\overline{Y_{n}}=\widehat{\sigma}(h)
\quad
\text{and}
\quad
\lim_{N\to+\infty}\frac1{N}\sum_{n=1}^{N}Y_{n} =
\sigma(\{0\})\;.
$$
Hence, if $D$ is nice vdC, then,
$$
      \sigma(\{0\})^{2}\leq\limsup_{|d|\to\infty,d\in
      D}|\widehat{\sigma}(d)|\;.
$$
And the claim (N3) is verified.


\medbreak

One more natural question concerns the Ramsey property.

Using product dynamical systems, it is easy to verify that the class
of sets of recurrence and the class of sets of strong recurrence
have the Ramsey property. We saw that the class of vdC sets and the class of  enhanced vdC sets have this property.
The other notions of vdC sets and of recurrence could be studied from
this point of view. 

\begin{question}
Do the class of sets of nice recurrence and the class of nice vdC sets have the Ramsey property ?
\end{question}

Note that the class of sets of nice recurrence has the Ramsey property if and only if the
following property of simultaneous nice recurrence is valid : given
any set $D\subset\Z\setminus\{0\}$ of nice recurrence, any m.p.s.
$(X,\A,\mu,T)$, any sets $A$ and $B$ in $\A$, and any $\epsilon>0$, there exists $d\in D$ such that
$$
\mu\left(A\cap T^{-d}A\right)>\mu(A)^2-\epsilon\quad\text{and}\quad
\mu\left(B\cap T^{-d}B\right)>\mu(B)^2-\epsilon\;.
$$

\section{Variations on the averaging method}\label{dist-notion}
In this short final section we provide additional remarks on some of the possible variations on the vdC theme which are related to different notions of averaging which naturally appear in the theory of uniform distribution and ergodic theory. For simplicity and in order to be able to 
more easily stress the important points, we restrict our discussion to subsets of $\Z$. 
We do want, however, to remark that many of the results in this paper can be extended to 
much a wider setup involving general groups and various methods of summation. 
(See for example \cite{Peres}, where some directions of extensions are indicated.)

\subsection{Well distribution}
Recall that a sequence $(x_n)_{n\in\N}$ of real numbers is \emph{well distributed} mod 1 if, for any continuous function $f$ on the torus $\T$, we have
$$
\lim_{N-M\to+\infty}\frac1{N-M}\sum_{n=M}^{N-1} f(x_n) = \int_{\T} f(t)\, \text{d}t.
$$
To this notion of well distribution is naturally associated a notion of van der Corput set. Let us call it w-vdC set: a set $D$ of positive integers is a w-vdC set if, for any sequence $(u_{n})_{n\in\N}$ of complex numbers
of modulus 1 such that
$$
\forall d\in D,\quad\gamma(d):=\lim_{N-M\to+\infty}
\frac1{N-M}\sum_{n=M}^{N-1}u_{n+d}\overline{u_{n}}=0
$$
we have
$$
\lim_{N-M\to+\infty} \frac1{N-M}\sum_{n=M}^{N-1}u_{n}=0\;.
$$
\\

The spectral characterization of vdC sets given in Theorem \ref{spectralchar} immediately implies that any vdC set is a w-vdC set.

But the proof, coming from Ruzsa (\cite{R}), of the fact that spectral properties (S1) and (S2) are necessary for vdC sets cannot be applied to w-vdC. This comes from the fact that the law of large numbers fails dramatically when we replace averages $1/N\sum_{0\leq n<N}$ by moving averages $1/(N-M)\sum_{M\leq n<N}$.

\begin{question} Is every w-vdC set a vdC set ?\end{question}

\subsection{F\o lner sequences}
Let $F=(F_N)_{N\geq1}$ be a F\o lner sequence in the space of parameters (which in this section is $\Z$). Let us say that a real sequence $(x_n)_{n\in\Z}$ is $F$-u.d. mod 1 
if, for any continuous function $f$ on the torus $\T$, we have
\begin{equation}\label{f-conv}
\lim_{N\to+\infty}\frac1{\left|F_N\right|}\sum_{n\in F_N}f(x_n)=\int_{\T}f(x)\,\text{d}x\;.
\end{equation}
(We say that the sequence $(f(x_n))$ converges to the integral of $f$ in the $F$-sense when (\ref{f-conv}) is satisfied.)

One can naturally define also the notion of $F$-vdC. A set $D$ of non zero integers is $F$-vdC set if any sequence $(x_n)$ such that, for all $d\in D$, the sequence $x_{n+d}-x_n$ is $F$-u.d. mod 1, is itself $F$-u.d. mod 1.

In order to compare the notion of $F$-vdC set with the notion of vdC set, it would be of interest to obtain a spectral characterization of $F$-vdC sets similar to Theorem \ref{spectralchar}.

Note that the sequence of correlations
$$
\gamma(h):=\limsup_{N\to+\infty}\frac1{\left|F_N\right|}\sum_{n\in F_N}u_{n+h}\overline{u_n}
$$
is positive-definite, and the F\o lner property is exactly what is needed in order to prove a 
result similar 
to Lemma \ref{class_spect}. An argument similar to the one used in the proof of implication 
(S2)$\Rightarrow$(S3) 
allows one to establish the fact that \emph{any vdC-set is an
$F$-vdC set}.

In the other direction we don't know any general result, but, keeping in mind the argument we used in the proof of Theorems \ref{spectralchar} and \ref{enhanced-sp.char}, we can state the following sufficient condition : \emph{suppose that for any probability measure on the torus $\T$ there exists a sequence $(Y_n)_{n\in\N}$ of complex numbers of modulus one such that, for all $h\in\Z$,
$$
\lim_{N\to+\infty}\frac1{\left|F_N\right|}\sum_{n\in F_N}Y_n=\sigma(\{0\})\quad\text{and}\quad\lim_{N\to+\infty}\frac1{\left|F_N\right|}\sum_{n\in F_N}Y_{n+h}\overline{Y_n}=\widehat\sigma(h);
$$
then any $F$-vdC set is a $vdC$ set.}

We have in particular the following result (and its multiparameter extensions).
\begin{proposition}If a F\o lner sequence $F$ is such that any bounded sequence which 
converges in the Ces\`aro sense also converges in the $F$-sense\footnote{If any bounded 
sequence which converges in the Ces\`aro sense also converges in the $F$-sense then the 
limits in the Ces\`aro sense and in the $F$-sense coincide (when they exist). This fact 
is left as an exercise for the reader.} then the notions of vdC set and $F$-vdC 
set coincide.
\end{proposition}

\section{Appendix. A remark on divisibility of polynomials}\label{appendix}
{\underline{Definitions}}. \begin{itemize}\item
A polynomial $p\in\Z[X]$ is {\it divisible} by an integer $d$ if 
there exits $n\in\Z$ such that $d$ divides $p(n)$. \item
A polynomial 
$p\in\Z[X]$ is {\it divisible} if it is divisible by any integer.
\item
Polynomials $p_1,p_2,\ldots, p_r\in\Z[X]$ are {\it simultaneously divisible} 
by an integer $d$ if 
there exists $n\in\Z$ such that $d$ divides $p_i(n)$, $1\leq i\leq r$. 
\item
Polynomials $p_1,p_2,\ldots, p_r\in\Z[X]$ are {\it simultaneously divisible} if 
they are simultaneously divisible by any integer.
\end{itemize}
\medbreak
(Trivial examples : if $p(0)=0$ then $p$ is divisible ; the polynomial 
$2X+1$ is not divisible ; polynomials $X$ and $X+1$ are divisible but not 
simultaneously divisible.)
\\
\medbreak
{\underline{Known facts}}.  Let $p_1,p_2,\ldots, p_r\in\Z[X]$. 
There is equivalence between the following assertions\\
\begin{itemize}
\item The sequence $(p_1(n),p_2(n),\ldots,p_r(n))_{n\in\N}$ is a Poincar\'e recurrence 
sequence for finite measure preserving $\Z^r$ actions;\\
\item The sequence $(p_1(n),p_2(n),\ldots,p_r(n))_{n\in\N}$ is a van der Corput sequence 
in $\Z^r$;\\
\item $p_1,p_2,\ldots, p_r$ are simultaneously divisible.
\end{itemize}
\medbreak
In \cite{BLL}, we prove that the simultaneous
divisibility of polynomials $p_1,p_2,\ldots, p_r$ is also a necessary and sufficient condition for multiple recurrence of the type
$$
\mu(A\cap T^{p_1(n)}A\cap T^{p_2(n)}A\cap\ldots \cap T^{p_r(n)}A)>0\;.
$$
\medbreak
{\underline{Claim}}. The simultaneous divisibility of a family of polynomials is a property strictly stronger than the divisibility of any of their linear combinations. In other words, there exist two polynomials $p$ and $q$ in $\Z[X]$
such that, for any integers $a$ and $b$, the 
polynomial $ap+bq$ is divisible but the polynomials $p$ and $q$ are not 
simultaneously divisible. 
\medbreak
Here are two facts which seem to go against the previous claim.
Let $p,q\in\Z[X]$.
\begin{itemize}\item
Let $ d $ be a prime number. If for all pairs $ (a,b) $ of integers, 
the polynomial $ ap+bq $ is divisible by $ d $, 
then $ p $ and $ q $ are simultaneously divisible by $ d $.
\item
Let $ d $ and $ e $ be two relatively prime integers. 
If $ p $ and $ q $ are simultaneously divisible by $ d $ 
and simultaneously divisible by $ e $, then they are
simultaneously divisible by $ de $.
\end{itemize}
These facts indicate that the key to the distinction between the simultaneous divisibility and the divisibility of linear combinations of polynomials lies with the divisibility by $d^k$ where $d$ is a prime number and $k>1$.

\medbreak
{\underline{Proof of the Claim}}. Let us show that the polynomials 
$$
p(X)=(2+X^2+X^3)(1+2X)\quad\text{and}\quad q(X)=X(1+X)(1+2X)
$$
are not simultaneously divisible by $4$ although the polynomial 
$ap+bq$ is divisible for all $a,b$ in $\Z$.

Modulo 4, we have $p(0)=2$ and $q(0)=0$, $p(1)=0$ and $q(1)=2$, 
$p(2)=q(2)=2$, $p(3)=2$ and $q(3)=0$. This shows that $p$ and $q$ are 
not simultaneously divisible by $4$.

Let us fix $a$ and $b$ in $\Z$ and show that $ap+bq$ is divisible. It 
is of course enough to consider the case when $a$ and $b$ are 
relatively prime. The divisibility of $ap+bq$ by odd integers is 
directly given by the presence of the common factor $1+2X$. Let us 
examine divisibility by the powers of 2. We will distinguish the case 
when one of the two numbers $a$ and $b$ is even, and the case when 
both are odd. 

First case : $a$ or $b$ is even (and the other is odd). Let us 
show by induction on $k$ that, for all $k\geq0$, there exists an odd 
number $n_{k}$ such that $2^k\mid ap(n_{k})+bq(n_{k})$. We can choose 
any number $n_{0}$, and $n_{1}=1$ is OK. Suppose that the result is 
true for an integer $k\geq1$. Define $\ell:=\max\{i\geq k\,:\,2^i\mid 
ap(n_{k})+bq(n_{k})\}$. We have $\ell\geq k$ and 
$ap(n_{k})+bq(n_{k})=2^\ell\alpha$, with $\alpha$ odd. Define a new odd number by
$n_{k+1}=n_{k}+2^\ell$. Using
$$
ap(X)+bq(X)=2aX^4+(3a+2b)X^3+(a+3b)X^2+(4a+b)X+2a\;,
$$
we note that, modulo $2^{\ell+1}$,
\begin{multline*}
ap(n_{k+1})+bq(n_{k+1})=\\
ap(n_{k})+bq(n_{k})+2a(4\cdot 2^\ell
n_{k}^3)+(3a+2b)(3\cdot2^\ell
n_{k}^2)+(a+3b)(2\cdot2^\ell n_{k})+(4a+b)2^\ell=\\
2^\ell\alpha+a2^\ell n_{k}^2+2^\ell b=2^\ell(\alpha+an_{k}^2+b)\;,
\end{multline*}
Since $\alpha+an_{k}^2+b$ is even, this shows that $2^{\ell+1}\mid
ap(n_{k+1})+bq(n_{k+1})$. We have $\ell+1\geq 
k+1$, and $n_{k+1}$ is odd. This concludes the induction.

Second case : $a$ and $b$ are odd.  Let us 
show by induction on $k$ that, for all $k\geq0$, there exists an even 
number $n_{k}$ such that $2^k\mid ap(n_{k})+bq(n_{k})$. We can choose 
any number $n_{0}$, and $n_{1}=2$ is OK. Suppose that the result is 
true for an integer $k\geq1$. We define $\ell$ and 
$n_{k+1}=n_{k}+2^\ell$ as in the first case, but now the number 
$n_{k}$ is even, hence we have still  
$$
ap(n_{k+1})+bq(n_{k+1})= 
2^\ell(\alpha+an_{k}^2+b)=0\quad\text{modulo $2^{\ell+1}$}\;,
$$
and the induction process works.

In any case, we have proved that $ap+bq$ is divisible by all the powers of 2.
We know also that the polynomial $ap+bq$ is divisible by any odd 
integer. Let us prove that it is divisible by 
any integer $2^k\alpha$, where $\alpha$ is odd. 
We write $ap(X)+bq(X)=(2X+1)r(X)$. We know that $2^k\mid r(n_{k})$. By 
the B\'ezout identity, there exist integers $u$ and $v$ such that 
$$
2n_{k}+1=-u2^{k+1}+v\alpha\;.
$$
We have $\alpha\mid 2(n_{k}+2^ku)+1$ and $2^k\mid r(n_{k}+2^ku)$, 
hence
$$
2^k\alpha\mid ap(n_{k}+2^ku)+bq(n_{k}+2^ku)\;.
$$
This proves that the polynomial $ap+bq$ is divisible.

\end{document}